\documentclass[12pt]{amsart}
\usepackage[margin=1in]{geometry}  % set the margins to 1in on all sides
\usepackage{amscd,latexsym,amsthm,amsfonts,amssymb,amsmath,amsxtra}
\usepackage{mathrsfs}
\usepackage[mathscr]{eucal}
\usepackage{hyperref}
\usepackage{pdfsync}
\usepackage[all,cmtip]{xy}

\numberwithin{equation}{section}
\newtheorem{thm}{Theorem}[section]

\newtheorem{prop}[thm]{Proposition}
\newtheorem{lemma}[thm]{Lemma}
\newtheorem{conj}[thm]{Conjecture}
\newtheorem{rk}[thm]{Remark}

\newtheorem{defn}[thm]{Definition}

\newcommand{\BA}{{\mathbb {A}}}

\newcommand{\BC}{{\mathbb {C}}}

\newcommand{\BN}{{\mathbb {N}}}

\newcommand{\BR}{{\mathbb {R}}}

\newcommand{\CS}{{\mathcal {S}}}

\newcommand{\RG}{{\mathrm {G}}}

\newcommand{\RL}{{\mathrm {L}}}

\newcommand{\RU}{{\mathrm {U}}}

\renewcommand{\Re}{{\mathrm {Re}}}

\newcommand{\fg}{\mathfrak{g}}

\newcommand{\ft}{\mathfrak{t}}

\newcommand{\GL}{{\mathrm{GL}}}
\newcommand{\SL}{{\mathrm{SL}}}
\newcommand{\SU}{{\mathrm{SU}}}

\newcommand{\SO}{{\mathrm{SO}}}

\newcommand{\Mp}{{\mathrm{Mp}}}

\newcommand{\quo}{\backslash}

\newcommand{\Bil}{{\mathrm{Bil}}}

\newcommand{\Ind}{{\mathrm{Ind}}}

\newcommand{\Ad}{{\mathrm{Ad}}}

\newcommand{\Lie}{{\mathrm{Lie}}}

\newcommand{\eps}{\epsilon}

\newcommand{\sbst}{\subseteq}
\newcommand{\norm}[1]{\lVert#1\rVert}
\newcommand{\abs}[1]{\lvert#1\rvert}
\newcommand{\set}[2]{\{#1\,|\,#2\}}
\newcommand{\bigset}[2]{\Biggl\{#1\,\bigg\lvert\,#2\Biggr\}}

\newcommand{\mtrtwo}[4]{\begin{pmatrix} #1 &#2 \\#3 &#4 \end{pmatrix}}
\newcommand{\mtrthr}[9]{\begin{pmatrix} #1 &#2 &#3 \\#4 &#5 &#6\\ #7 &#8 &#9 \end{pmatrix}}

\renewcommand{\bar}{\overline}
\renewcommand{\tilde}{\widetilde}

\begin{document}

\title[Archimedean Local Gamma Factors for $\Ad(\GL_3)$, Part I]{On the Archimedean Local Gamma Factors for Adjoint Representation of $\GL_3$, Part I}

\author{Fangyang Tian}
\address{Department of Mathematics\\
National University of Singapore, Singapore}
\email{mattf@nus.edu.sg}

\subjclass[2010]{Primary 22E45, 22E50; Secondary 11F70}

\date{\today}

\keywords{Archimedean Local Integral, Rankin-Selberg Integral, Adjoint $L$-function for General Linear Group}

\begin{abstract}
  Studying the analytic properties of the partial Langlands $L$-function via Rankin-Selberg method has been proved to be successful in various cases. Yet in few cases is the local theory studied at the archimedean places, which causes a tremendous gap to complete the analytic theory of the complete $L$-function. In this paper, we will establish the meromorphic continuation and the functional equation of the archimedean local integrals associated with D. Ginzburg's global integral (\cite{Gin}) for the adjoint representation of $\GL_3$. Via the local functional equation, the local gamma factor $\Gamma(s,\pi,\Ad,\psi)$ can be defined. In a forthcoming paper, we will compute the local gamma factor $\Gamma(s,\pi,\Ad,\psi)$ explicitly, which fill in some blanks in the archimedean local theory of Ginzburg's global integral.
\end{abstract}

\maketitle
\tableofcontents

%%%%%%%%%%%%%%%%%%%%%%%%%%%%%%%%%%%%%%%%%%%%%%%%%%%%%%%%%%%%%%%%%%%%%%%%%%%%%%%%%%%%%%%%%%%%%%%%%%%%%%%%%%%%%%%%%%%%%%%%%

\section{Introduction and Notations}\label{Section: Intro}

%%%%%%%%%%%%%%%%%%%%%%%%%%%%%%%%%%%%%%%%%%%%%%%%%%%%%%%%%%%%%%%%%%%%%%%%%%%%%%%%%%%%%%%%%%%%%%%%%%%%%%%%%%%%%%%%%%%%%%%%%

   \subsection{The Adjoint $L$-function for $\GL_n$}\label{subsection: Adjoint L function}

%%%%%%%%%%%%%%%%%%%%%%%%%%%%%%%%%%%%%%%%%%%%%%%%%%%%%%%%%%%%%%%%%%%%%%%%%%%%%%%%%%%%%%%%%%%%%%%%%%%%%%%%%%%%%%%%%%%%%%%%%%
   The study of analytic properties of $L$-functions can be dated back to B. Riemann's time, when the meromorphic continuation and functional equation of Riemann Zeta function $\zeta(s)$ were established via integral representations of Jacobi-Theta series. The analytic theory of Riemann Zeta function $\zeta(s)$, or in general the Dirichlet $L$-function $L(s,\chi)$, can now be interpreted as theory of automorphic forms of the group $\GL_1$ (see \cite{TateThesis}). When we move to groups of higher rank, we consider a connected linear algebraic group $\RG$ defined over a number field $F$. We write $\mathbb{A}$ for the ring of adeles of $F$. An automorphic representation $\pi$ of $\RG(\BA)$ is a natural non-abelian generalization of a Hecke character of $\GL_1$. Correspondingly, we will study the Langlands $L$-function $L(s,\pi,r)$, a natural generalization of the Dirichlet $L$-function, where $r$ is a finite dimensional representation of the $L$-group ${}^\RL\RG$. It is a theorem of Langlands that $L(s,\pi,r)$ converges on some right half plane. Langlands also conjectured that
   \begin{conj}
      The L-function $L(s,\pi,r)$ has a meromorphic continuation to the whole complex plane which has only finitely many poles. It also satisfies a functional equation relating $s$ to $1-s$.
   \end{conj}
   We refer to \cite{BorelAutomorphicLFunction} for more detailed discussion on Langlands $L$-functions. Among all the Langlands $L$-functions, the adjoint $L$-function is of great interest. The adjoint $L$-function, in particular its special value at 1, appears in computations of automorphic period, in the local theory related to Plancherel measure, etc. (see \cite{HarrisAdjointMotive},\cite{H-I-I FormalDegreeAdjointGamma}, \cite{VenkateshTakagiLecture}). We also refer to \cite{JiangRallisFourierCoeffiG2} for an interesting and deep relation between the holomorphy of the adjoint $L$-function and a classical algebraic number theoretic conjecture on holomorphy of the quotient of two Dedekind Zeta functions.

   The Rankin-Selberg method, which was originally developed to study the standard $L$-function for $\GL_2\times \GL_2$, has been proved to be successful in providing fruitful results on the analytic properties of Langlands $L$-function. We refer to D. Bump's paper \cite{BumpRankinIntegralSurvey} for an excellent survey on this topic. The study of adjoint $L$-function via Rankin-Selberg method started from S. Gelbart and H. Jacquet's work on $\GL_2$ (see \cite{GelbartJacquetGL2Adjoint}), where they studied the global and local theory of the global integral (due to G. Shimura) constructed by integrating a cusp form on $\GL_2$ against an Eisenstein series on the Metaplectic group $\Mp_2$. As for the adjoint $L$-function for $\GL_3$, D. Ginzburg constructed in his paper \cite{Gin} a global integral by integrating a cusp form on $\GL_3$ against an Eisenstein series on $\RG_2$. He only proved that the global integral admits an Euler product and at the local unramified non-archimedean places, the local integral is the adjoint $L$-function $L(3s-1,\pi,\Ad)$ (modulo some normalizing factor). His construction is quite remarkable and subtle. We will review Ginzburg's work with more details in the next two Subsections. To complete the analytic theory of the adjoint $L$-function of $\GL_3$, we also have to take care of the local archimedean places as well as the non-archimedean places where local representation is ramified. The goal of this paper is to explore some basic archimedean local theory and establish a foundation for attacking the local functional equation at archimedean places in a forthcoming paper. The construction of the global integrals for the adjoint $L$-function for $\GL_4$ and $\GL_5$ uses similar ideas, i.e. integrating a cusp form on $\GL_4$ ($\GL_5$ resp.) against an Eisenstein series on the exceptional group $F_4$ ($E_8$ resp.). We refer to \cite{BumpGinzburgAdjointGL4} and \cite{GinzburgHundleyAdjointGL5} for detailed construction in these two cases. Unfortunately, less is known on the local theory in the case of $\GL_5$. For example, even the unramified computation in \cite{GinzburgHundleyAdjointGL5} is not complete.

%%%%%%%%%%%%%%%%%%%%%%%%%%%%%%%%%%%%%%%%%%%%%%%%%%%%%%%%%%%%%%%%%%%%%%%%%%%%%%%%%%%%%%%%%%%%%%%%%%%%%%%%%%%%%%%%%%%%%%%%%%%%%%%%%%%%%%%%%%%%%%

   \subsection{Some Basic Structures of $\RG_2$}\label{section: structure G2}

%%%%%%%%%%%%%%%%%%%%%%%%%%%%%%%%%%%%%%%%%%%%%%%%%%%%%%%%%%%%%%%%%%%%%%%%%%%%%%%%%%%%%%%%%%%%%%%%%%%%%%%%%%%%%%%%%%%%%%%%%%%%%%%%%%%%%%%%%%%%%%
    Let us first review some basic structures of $\RG_2$. Our main reference is \cite{SGA3}. Let $\RG_2$ be the connected, simply connected group with Lie algebra $\fg_2$, and $T$ be the torus of $\RG_2$ with Lie algebra $\ft$. The root system of $\Phi(\fg_2,\ft)$ can be described as
    \begin{equation*}
       \Phi(\fg_2, \ft) = \set{\pm\eps_i, \eps_i-\eps_j}{i,j=1,2,3\quad \text{and}\quad i\ne j},
    \end{equation*}
    where $\eps_1+\eps_2+\eps_3 = 0$. We choose a positivity on the root system $\Phi(\fg_2,\ft)$ such that the set of simple roots is given by
    \begin{equation*}
       \{\alpha = \eps_2, \beta = \eps_1-\eps_2\},
    \end{equation*}
    i.e. $\alpha$ is the short simple root while $\beta$ is the long simple root. Hence all positive roots are $\alpha,\beta,\alpha+\beta,2\alpha+\beta,3\alpha+\beta,3\alpha+2\beta$.  All the six long roots $\pm\beta, \pm(3\alpha+\beta), \pm(3\alpha+2\beta)$ generate a subgroup of $\RG_2$ isomorphic to $\SL_3$. We emphasize that this subgroup is not a Levi subgroup of any parabolic subgroup of $\RG_2$. This subtle embedding will cause some major computational difficulties in the theory of archimedean local integrals.

    For any root $\gamma$, let $\fg_{2,\gamma}$ be the root space of $\fg_2$ corresponding to $\gamma$. Denote by $x_\gamma(t)$ the one parameter unipotent subgroup of $\RG_2$ with Lie algebra $\fg_{2,\gamma}$. For each simple root $\gamma\in\{\alpha,\beta\}$, there is a homomorphism $\phi_\gamma: \text{SL}_2\rightarrow \RG_2$ such that
    \begin{equation*}
       \phi_\gamma(\mtrtwo{1}{0}{t}{1}) = x_{-\gamma}(t), \quad \phi_\gamma(\mtrtwo{1}{t}{0}{1}) = x_\gamma(t).
    \end{equation*}
    Set
    \begin{equation*}
       w_\gamma = \phi_\gamma(\mtrtwo{0}{1}{-1}{0}).
    \end{equation*}
    Then it is easy to check that
    \begin{equation*}
       w_\gamma = x_\gamma(1)x_{-\gamma}(-1)x_\gamma(1),
    \end{equation*}
    and $w_\gamma$ is a representative for the simple reflection relative to simple root $\gamma$. Throughout this paper, when no confusion arises, we will not distinguish an element in the Weyl group from its representative. Following \cite[Expos\'{e} XXIII, Section 3.4]{SGA3}, for any root $\gamma$, since dim$\fg_{2,\gamma}=1$, we can choose generators $X_{\gamma}\in \fg_{2,\gamma}$ in the following way:
    \begin{equation}\label{eq: Root vector g2}
       \begin{aligned}
          &X_\alpha = dx_{\alpha}(1),\qquad  X_{-\alpha} = dx_{-\alpha}(1),\\
          &X_\beta = dx_{\beta}(1),\qquad  X_{-\beta} = dx_{-\beta}(1),\\
          &X_{\alpha+\beta} = \Ad(w_\beta)(X_\alpha),\quad X_{2\alpha+\beta} = \Ad(w_\alpha)(X_{\alpha+\beta}), \\
          &X_{3\alpha+\beta} = -\Ad(w_\alpha)(X_\beta),\quad X_{3\alpha+2\beta} = \Ad(w_\beta)(X_{3\alpha+\beta}).\\
       \end{aligned}
    \end{equation}
    For any negative root $\gamma$, we can construct generators $X_{\gamma}\in \fg_{2,\gamma}$ from $X_{-\alpha}, X_{-\beta}$ and simple reflections using the same recipe. Finally, for any root $\gamma$,  the one dimensional subgroup $x_\gamma(t)$ is parameterized such that $X_\gamma = dx_\gamma(1)$. Then the following commutator relations hold:
    \begin{lemma}\cite[Expos\'{e} XXIII, Section 3.4]{SGA3}\label{Lemma£ºCommutator Relations}
       We adopt the following convention of commutator $(x,y)=x^{-1}y^{-1}xy$, and have
       \begin{equation}
       \begin{aligned}
          (x_{\beta}(s),x_{\alpha}(t)) &= x_{\alpha+\beta}(st)x_{2\alpha+\beta}(st^2)x_{3\alpha+\beta}(st^3)x_{3\alpha+2\beta}(s^2t^3),\\
          (x_{\alpha+\beta}(s),x_{\alpha}(t)) &= x_{2\alpha+\beta}(2st)x_{3\alpha+\beta}(3st^2)x_{3\alpha+2\beta}(3s^2t),\\
          (x_{2\alpha+\beta}(s),x_{\alpha}(t)) &= x_{3\alpha+\beta}(3st),\\
          (x_{3\alpha+\beta}(s),x_{\beta}(t)) &= x_{3\alpha+2\beta}(-st),\\
          (x_{2\alpha+\beta}(s),x_{\alpha+\beta}(t)) &= x_{3\alpha+2\beta}(3st).\\
       \end{aligned}
       \end{equation}
       For any other pair $(\gamma,\delta)$ of positive roots, $(x_\gamma(s),x_\delta(t))=1$.
    \end{lemma}
    Also, we have
    \begin{lemma}\cite[Expos\'{e} XXIII, Section 3.4]{SGA3}\label{Lemma: Adjoint Action of Lie Algebra}
       \begin{equation}
          \begin{aligned}
             &\Ad(w_\alpha)(X_{2\alpha+\beta}) = -X_{\alpha+\beta}, \qquad \Ad(w_\alpha)(X_{3\alpha+\beta})=X_{\beta},\\ &\Ad(w_\alpha)(X_{3\alpha+2\beta})=X_{3\alpha+2\beta},\qquad\Ad(w_\beta)(X_{\alpha+\beta}) = -X_{\alpha},\\ &\Ad(w_\beta)(X_{2\alpha+\beta})=X_{2\alpha+\beta},\qquad \Ad(w_\beta)(X_{3\alpha+2\beta})=-X_{3\alpha+\beta}.\\
          \end{aligned}
       \end{equation}
    \end{lemma}
    \begin{rk}
        The choices of root vectors in this paper are slightly different from those of \cite{Gin}. For example, \cite{Gin} 1.1 implies that $\Ad(w_\beta)(X_{\alpha+\beta}) = X_{\alpha}$ under Ginzburg's choice of $X_{\alpha+\beta}$, but this contradicts to Lemma \ref{Lemma: Adjoint Action of Lie Algebra}. We note that our choice and Ginzburg's choice only differ by a sign.
    \end{rk}

%%%%%%%%%%%%%%%%%%%%%%%%%%%%%%%%%%%%%%%%%%%%%%%%%%%%%%%%%%%%%%%%%%%%%%%%%%%%%%%%%%%%%%%%%%%%%%%%%%%%%%%%%%%%%%%%%%%%%%%%%%%%%%%%%%%%%%%%%%%%%%

   \subsection{Ginzburg's Work on $L(s,\pi,\Ad)$ and Related Work of Other People}

%%%%%%%%%%%%%%%%%%%%%%%%%%%%%%%%%%%%%%%%%%%%%%%%%%%%%%%%%%%%%%%%%%%%%%%%%%%%%%%%%%%%%%%%%%%%%%%%%%%%%%%%%%%%%%%%%%%%%%%%%%%%%%%%%%%%%%%%%%%%%%
   In this Subsection, we continue to assume that $F$ is a number field and retain all notations in Subsection \ref{section: structure G2}. Let $P$ be the standard maximal parabolic subgroup of $\RG_2$ with Levi decomposition is $P=MU$, where $U$ is generated by the following one parameter subgroups $$x_\beta(t), x_{\alpha+\beta}(t), x_{2\alpha+\beta}(t), x_{3\alpha+\beta}(t), x_{3\alpha+2\beta}(t).$$ In \cite{Gin}, Ginzburg considered  the normalized parabolically induced representation
   \begin{equation}\label{eq: global induced repn G2}
       \rho_s := \Ind_{P(\mathbb{A})}^{\RG_2(\mathbb{A})}\delta_P^{s-\frac{1}{2}}
   \end{equation}
   on the space $V_{\rho_s}$. From Langlands' theory of Eisenstein series (see \cite{Sha}), if we take a function $f_s\in V_{\rho_s}$, then the Eisenstein series
   \begin{equation}\label{eq: eisenstein series}
       E(g,f_s) := \sum_{\gamma\in P(F)\quo \RG_2(F)} f_s(\gamma g)
   \end{equation}
   can be extended to a meromorphic function of $s\in \BC$. It satisfies a functional equation
        \begin{equation*}
           E(g, f_s) = E(g, M(w)f_s),
        \end{equation*}
   where $M(w)$ is the intertwining operator relative to the Weyl element $w=w_{3\alpha+2\beta}$.

   Let $\pi$ be a cuspidal representation of $\GL_3(\BA)$ with a trivial central character. Then $\pi$ is generic and has a global Whittaker model $\mathcal{W}(\pi,\psi)$ for some non-trivial character $\psi$ of $F\quo \BA$. Here the generic character for the standard maximal unipotent subgroup $N$ of $\GL_3$ is given by
   \begin{equation}\label{eq: global generic char}\psi_N(\mtrthr{1}{x}{z}{}{1}{y}{}{}{1}) = \psi(x+y).\end{equation}
   The global integral which Ginzburg constructed is obtained by integrating a cusp form $\phi \in V_\pi$ against the Eisenstein series $E(g, f_s)$ (in \eqref{eq: eisenstein series}):
   \begin{equation}\label{eq: global integral}
       Z(\phi, f_s) = \int_{\SL_3(F)\quo \SL_3(\mathbb{A})} \phi(g)E(g, f_s) dg,
   \end{equation}
   Ginzburg proved in \cite[Theorem 1]{Gin} that the above global integral $Z(\phi, f_s)$ admits an Euler product
   \begin{equation}\label{eq: Euler product}
       Z(\phi, f_s) = \prod_\nu Z_\nu(W_{\phi_\nu}, f_{s,\nu}).
   \end{equation}
   Here the local integrals $Z_\nu(W_{\phi_\nu}, f_{s,\nu})$ are defined by
   \begin{equation}\label{Def: local integral}
           Z_\nu(W_{\phi_\nu}, f_{s,\nu}) = \int_{N_2(F_\nu)\quo \SL_3(F_\nu)}  W_{\phi_\nu}(g)f_{s,\nu}(\gamma g)dg,
   \end{equation}
   where $W_{\phi_\nu}$ lives in the local Whittaker model $\mathcal{W}(\pi_\nu,\psi_\nu)$, $\gamma = x_{-(\alpha+\beta)}(-1)w_\beta$ is the representative of the open orbit in $P\quo \RG_2/\SL_3$, and the group $N_2$ consists of all matrices of the form
   \begin{equation}\label{eq: group N2}
                       \mtrthr{1}{x}{z}{}{1}{-x}{}{}{1}.
   \end{equation}
   Ginzburg also proved in \cite[Theorem 2]{Gin} that when $\nu$ is a non-archimedean local place where $\pi_\nu$ and $\rho_{s,\nu}$ are unramified, if we take $W_{\phi_\nu}$ to be the spherical Whittaker function and $f_{s,\nu}\in V_{\rho_{s,\nu}}$ to be the spherical function such that
   \begin{equation*}
           W_{\phi_v}(I) = f_{s,\nu}(I) = 1,
   \end{equation*}
   the local integral can be explicitly computed in terms of the local unramified adjoint L-function of $\pi_\nu$, with some normalizing factors:
   \begin{equation}\label{Eq: local unramified computation 1}
          Z_\nu(W_{\phi_\nu}, f_{s,\nu}) = \frac{L(3s-1,\pi_\nu,\text{Ad})}{\zeta_\nu(3s)\zeta_\nu(6s-2)\zeta_\nu(9s-3)}.
   \end{equation}

   A modification of Ginzburg's construction for adjoint $L$-function of unitary groups was first found by J. Hundley (see \cite{HundleyAdjointSU21}). By studying the local and global theory of these integral representations (for $\GL_3$ and unitary groups), Hundley and Q. Zhang (see \cite{HundleyZhangAdjointGL3}) were able to extend the result of holomorphy of partial $L$-function $L^S(s,\pi,\Ad)$ obtained in \cite{HundleyAdjointA2} to the finite part of the complete $L$-function, denoted by $L_f(s,\pi,\Ad)$. Thus, one should be able to complete the analytic theory of complete $L$-function $L(s,\pi,\Ad)$ for $\GL_3$ and $\RU(2,1)$ via integral representations once the corresponding archimedean local theory is complete. This is one of the main intuitions of this paper.

%%%%%%%%%%%%%%%%%%%%%%%%%%%%%%%%%%%%%%%%%%%%%%%%%%%%%%%%%%%%%%%%%%%%%%%%%%%%%%%%%%%%%%%%%%%%%%%%%%%%%%%%%%%%%%%%%%%%%%%%%%%%%%%%%%%%%%%%%%%%%%%%%

   \subsection{Archimedean Local Setup and Organization of this Paper}\label{subsection: organization }

   The goal of this paper is to study the archimedean local theory for the integrals \eqref{Def: local integral}. Let us clarify the setup in the archimedean local theory first.

   For the rest of the paper, unless we state that $F$ is non-archimedean, we will always assume that $F$ is an archimedean local field, i.e. $F = \BR$ or $\BC$. Though the archimedean component of a cuspidal representation is in prior a $(\fg,K)$-module, for some technical reasons, it is more convenient to work with their Casselman-Wallach completions. Recall that a representation of a real reductive group is a Casselman-Wallach representation, if it is a smooth Fr\'{e}chet representation which satisfies the moderate growth condition and its Harish-Chandra module has finite length. For the constructions and basic properties of Casselman-Wallach representations, we refer to \cite[Chapter 11]{Wal2}. Thus, following Ginzburg's assumption, we assume that $\pi$ is an irreducible admissible generic Casselman-Wallach representation of $\GL_3(F)$ with a trivial central character. Then $\pi$ has a non-zero Whittaker model $\mathcal{W}(\pi,\psi)$ for some non-trivial additive character $\psi$ of $F$. The generic character $\psi_N$ associated with $\psi$ is defined exactly in the same way as in \eqref{eq: global generic char}. The local archimedean analogue of induced representation in \ref{eq: global induced repn G2} is
   \begin{equation}\label{eq: local induced repn G2}
       \rho_s := \Ind_{P(F)}^{\RG_2(F)}\delta_P^{s-\frac{1}{2}},
   \end{equation}
   which is assumed to be a Casselman-Wallach representation. The representation space of the above $\rho_s$ is still denoted by $V_{\rho_s}$. The modular character $\delta_P$ satisfies
   \begin{equation*}
       \delta_P(\mtrthr{t_1}{}{}{}{t_2}{}{}{}{t_3}) = \abs{t_1^3t_3^{-3}}_{F}.
    \end{equation*}
    The local integral which we consider in this paper is
   \begin{equation}\label{Def: local archi integral}
           Z(W_{v}, f_{s}) = \int_{N_2(F)\quo \SL_3(F)}  W_{v}(g)f_{s}(\gamma g)dg,
   \end{equation}
   where $\gamma = x_{-(\alpha+\beta)}(-1)w_\beta$.

   We first prove that
   \begin{thm}\label{Theorem: Absolute Convergence of Integral}
       When $F$ is archimedean, there is a sufficiently large $s_0$ such that when $\Re(s)> s_0$, the local integrals $Z(W_v,f_s)$ defined in (\ref{Def: local archi integral}) converge absolutely.
   \end{thm}
   The next step is to establish a meromorphic continuation of the local integrals $Z(W_v,f_s)$. To achieve this, we first need to establish an asymptotic expansion of the Whittaker function $W_v$  along the torus of $\GL_3(F)$. In this paper, there are two asymptotic expansions of this nature. The first one is a coarse asymptotic expansion due to H. Jacquet and J. Shalika in Section \ref{section: convergence} (see \eqref{Eq: Course Asymp}). Unfortunately, this coarse asymptotic expansion cannot help us to see the dependence of the each summand in the expansion of $W_v$ on the choice of vector $v$. A refined asymptotic expansion is suggested by D. Soudry (see \cite{Sou}), following N. Wallach's method (see \cite[Section 15.2]{Wal2}). Moreover, if we consider a character $\eta$ of the standard Borel subgroup $B_{\GL_3}$ of $\GL_3(F)$, defined by
    \begin{equation}\label{eq: char eta}
        \eta(\mtrthr{t_1}{\star}{\star}{}{t_2}{\star}{}{}{t_3}) := \prod_{i=1}^3\abs{t_i}_F^{u_i}\chi_i(t_i),
    \end{equation}
    where each $u_i\in \BC$ and each $\chi_i$ is a unitary character of the maximal compact subgroup of $F$; and consider a principal series
    \begin{equation}\label{eq: principal series}\pi = \Ind_{B_{\GL_3}}^{\GL_3(F)}\eta,\end{equation}
    then we also write $\pi = \pi_u$, $W_v = W_{v,u}$ for the analytic continuation of the Jacquet integral (see \cite[15.4.1]{Wal2}), to emphasize their dependence on the complex parameters $u = (u_1,u_2,u_3) \in \BC^3$. We can also keep track on the dependence of the asymptotic expansion of $W_{v,u}$ on the complex parameters $u$. Soudry proved these results for orthogonal groups in \cite{Sou}. He also claims that his approach works for all split real reductive groups. Thus, in Section \ref{section: aysmp}, we will follow Soudry's method closely and carry out the detailed proof for $\GL_3(F)$.

    With the refined asymptotic expansion at hand, we are able to establish the meromorphic continuation of the local integrals $Z(W_v,f_s)$ in Section \ref{section:  Mero}. Now we state the theorem.
   \begin{thm}\label{Theorem: Meromorphic Continuation}
       When $F$ is archimedean, the local integral $Z(W_v,f_s)$ extends to a meromorphic function of $s$ on the whole complex plane. Under the projective tensor product topology, $Z(W_v, f_s)$ is a continuous bilinear form on $V_\pi\hat{\otimes} V_{\rho_s}$ (the projective tensor space of two Fr\'{e}chet spaces $V_\pi$ and $V_{\rho_s}$). Moreover, if $\pi = \pi_u$ is a principal series as in \eqref{eq: principal series}, then $Z(W_{v,u}, f_s)$ is also meromorphic in $u$.
   \end{thm}

   To establish the functional equation, we will also consider a local integral $\tilde{Z}(W_v,f_s)$ obtained by applying the local intertwining operator $M(w_{3\alpha+2\beta})$ to $f_s$. Here, the local intertwining operator is defined by
    \begin{equation*}
       (M(w_{3\alpha+2\beta})f_s)(g) = \int_U f_s(w_{3\alpha+2\beta}^{-1}ug)du
    \end{equation*}
   when $\Re(s)$ is sufficiently large. It has a meromorphic continuation in $s$. The local integrals $\tilde{Z}(W_v,f_s)$ on the other side of the local functional equations are
    \begin{equation*}
       \tilde{Z}(W_v,f_s) := \int_{N_2(F)\quo \text{SL}_3(F)} W_v(g)\cdot (M(w_{3\alpha+2\beta})f_s)(\gamma\cdot g)dg.
    \end{equation*}
    Then $\tilde{Z}(W_v,f_s)$ also has a meromorphic continuation in $s$ and is continuous on $V_{\pi}\hat{\otimes} V_{\rho_s}$ for all parameter $s$, according to Theorem \ref{Theorem: Meromorphic Continuation}. Now we state the Uniqueness Theorem.
    \begin{thm}\label{Theorem: Uniqueness Theorem}
       When $F$ is archimedean, given any irreducible Casselman-Wallach representation $\pi$ whose central character is trivial, there is a discrete, at most countable subset $S$ of $\BC$ depending on $\pi$, such that whenever $s\notin S$, the space of continuous bilinear forms $B(v,f_s)$ satisfying the following equivariant property
       \begin{equation}\label{Eq: Equivariant Property}
          B(\pi(g)v, \rho_s(g)f_s) = B(v, f_s), \quad \forall g\in \SL_3(F),
       \end{equation}
       is at most one dimensional.
    \end{thm}
    We will prove the above theorem in Section \ref{section: uniqueness thm}. Now that $Z(W_v,f_s)$ and $\tilde{Z}(W_v,f_s)$ are both continuous bilinear forms on $V_\pi\hat{\otimes} V_{\rho_s}$ which are meromorphic in $s$ and satisfy the same equivariant property (\ref{Eq: Equivariant Property}). Thus, there exists a meromorphic function $\Gamma(s,\pi, \Ad, \psi)$ such that
   \begin{equation}\label{Eq: 184}
       \Gamma(s,\pi,\Ad,\psi)Z(W_v,f_s) = \tilde{Z}(W_v,f_s).
   \end{equation}
   We also note that if $\pi = \pi_u$ is a principal series, then $\Gamma(s,\pi_u, \Ad, \psi,)$ is also meromorphic in $u$.
   \begin{rk}
      An analogue of Theorem \ref{Theorem: Uniqueness Theorem} for non-archimedean cases was established in \cite[Proposition 3.6]{HundleyZhangAdjointGL3}, thus the local functional equation \eqref{Eq: 184} is in fact established at all places.
   \end{rk}

   We will also mention some results for non-archimedean local integrals in Subsection \ref{subsection: non-archi remark}, yet proofs will not be given here since we wish to focus on the archimedean local theory. In a forthcoming paper, we will compute the local gamma factors $\Gamma(s,\pi,\Ad,\psi)$ explicitly. To complete the local theory at archimedean cases, we also need to solve the g.c.d problem, i.e. the poles of the local integrals are governed by the $L$-function, and the L-function can be obtained by testing some smooth (or even $K$-finite) functions. These work is in progress.

   Finally, I greatly thank my advisor Dihua Jiang for introducing me to this wonderful topic and Lei Zhang for some helpful advices on the current version of this paper.

%%%%%%%%%%%%%%%%%%%%%%%%%%%%%%%%%%%%%%%%%%%%%%%%%%%%%%%%%%%%%%%%%%%%%%%%%%%%%%%%%%%%%%%%%%%%%%%%%%%%%%%%%%%%%%%%%%%%%%%%%

\section{Convergence of Local Integrals}\label{section: convergence}

%%%%%%%%%%%%%%%%%%%%%%%%%%%%%%%%%%%%%%%%%%%%%%%%%%%%%%%%%%%%%%%%%%%%%%%%%%%%%%%%%%

    \subsection{Archimedean Case}

        Let us first start with the archimedean case and assume that $F= \BR$ or $\BC$ throughout this Subsection. We will retain all notations in Subsection \ref{subsection: organization }. We have to distinguish two absolute values $\abs{\quad}$ and $\abs{\quad}_F$ when $F=\BC$. The absolute value with the subscript is defined by $\abs{z}_\BC = z\bar{z}$, while the absolute value without the subscript is the ordinary one. The absolute value $\abs{\quad}$ will mainly appear in the Iwasawa decomposition of the elements in the lower unipotent subgroups of $\RG_2(F)$. We fix a maximal compact subgroup $K_{\RG_2}$ such that the following Iwasawa decomposition holds:
            \begin{equation}\label{Eq: Iwasawa for G2}
                 \RG_2(F) = N_{\RG_2}A_{\RG_2}K_{\RG_2},
            \end{equation}
        where $N_{\RG_2}$ is the standard maximal unipotent subgroup of $\RG_2(F)$ generated by all positive roots. Then $K_{\SL_3}$ = $\SL_3(F)\bigcap K_{\RG_2}(F)$ is a maximal compact subgroup of $\SL_3(F)$. Hence
            \begin{equation}\label{eq: Iwasawa decomposition}
                 K_{\SL_3} = \begin{cases} \SO_3(\BR) & \text{ if $F=\BR$},\\
                                \SU_3     & \text{ if $F=\BC$}.
                             \end{cases}
            \end{equation}

        Let us prove a simple lemma first.
        \begin{lemma}\label{lemma: archimedean lemma}
               For any function $f_s\in V_{\rho_s}$ (see \eqref{eq: local induced repn G2}), we have
        \begin{equation}\label{Eq: 06}
               \begin{aligned}
                     &f_s(\gamma \mtrthr{1}{z}{}{}{1}{}{}{}{1}\mtrthr{t_1t_2}{}{}{}{t_2}{}{}{}{t_1^{-1}t_2^{-2}} g) \\
                              = &\abs{t_1t_2^3}_F^{3s}\cdot\Big\lvert\abs{t_1^{-1}z}^2+(1+\abs{t_2}^2)^3\Big\lvert^{-\frac{3}{2}s}_F\cdot f_s(k''(-\frac{-t_1^{-1}z}{(1+\abs{t_2}^2)^{\frac{3}{2}}})k'(t_2)w_\beta g)
               \end{aligned}
        \end{equation}
        where $k'$ and $k''$ are smooth functions taking values in $K_{\RG_2}$.
        \end{lemma}
        \begin{proof}
        By Lemma \ref{Lemma£ºCommutator Relations}, $x_{-\beta}$ and $x_{-(\alpha+\beta)}$ commutes. Then the following computation holds:
                \begin{equation}\label{Eq: simple conjugation}
                      \gamma\cdot \mtrthr{1}{z}{}{}{1}{}{}{}{1}\cdot\mtrthr{t_1t_2}{}{}{}{t_2}{}{}{}{t_1^{-1}t_2^{-2}} = \mtrthr{t_2}{}{}{}{t_1t_2}{}{}{}{t_1^{-1}t_2^{-2}}x_{-\beta}(-t_1^{-1}z)x_{-(\alpha+\beta)}(-t_2)w_\beta.
                \end{equation}
        We have the following Iwasawa decomposition of $x_{-(\alpha+\beta)}(-t_2)$ in $\RG_2(F)$:
                \begin{equation}\label{Eq:05}
                       x_{-(\alpha+\beta)}(-t_2) = \mtrthr{(1+\abs{t_2}^2)^{-1}}{}{}{}{(1+\abs{t_2}^2)^{\frac{1}{2}}}{}{}{}{(1+\abs{t_2}^2)^{\frac{1}{2}}}x_{\alpha+\beta}(-\bar{t_2})k'(t_2),
                \end{equation}
        where $k'$ is a smooth function taking values in the maximal compact subgroup  $K_{\RG_2}$, and $\bar{t_2}$ is the complex conjugate of $t_2$. Then by \eqref{Eq: simple conjugation}, \eqref{Eq:05} and the invariance of $f_s$, we get
        \begin{equation*}
                    \begin{aligned}
                  &f_s(\gamma \mtrthr{1}{z}{}{}{1}{}{}{}{1}\mtrthr{t_1t_2}{}{}{}{t_2}{}{}{}{t_1^{-1}t_2^{-2}} g)\\
                       =& \abs{t_1t_2^3}_F^{3s} f_s(x_{-\beta}(-t_1^{-1}z)x_{-(\alpha+\beta)}(-t_2)w_\beta g)\\
                       =& \abs{t_1t_2^3}_F^{3s} f_s(x_{-\beta}(-t_1^{-1}z)\mtrthr{(1+\abs{t_2}^2)^{-1}}{}{}{}{(1+\abs{t_2}^2)^{\frac{1}{2}}}{}{}{}{(1+\abs{t_2}^2)^{\frac{1}{2}}}x_{\alpha+\beta}(-\bar{t_2})k'(t_2)w_\beta g)\\
                       =& \abs{t_1t_2^3}_F^{3s}\cdot \abs{1+\abs{t_2}^2}_F^{-\frac{9}{2}s}\cdot f_s(x_{-\beta}(\frac{-t_1^{-1}z}{(1+\abs{t_2}^2)^{\frac{3}{2}}})x_{\alpha+\beta}(-\bar{t_2})k'(t_2)w_\beta g).\\
                    \end{aligned}
        \end{equation*}
        By Lemma \ref{Lemma£ºCommutator Relations}, $x_{-\beta}(-\frac{t_1^{-1}z}{(1+\abs{t_2}^2)^{\frac{3}{2}}})x_{\alpha+\beta}(-\bar{t_2})x_{-\beta}(\frac{t_1^{-1}z}{(1+\abs{t_2}^2)^{\frac{3}{2}}})$ belongs to the maximal unipotent subgroup of $\RG_2(F)$. Hence
        \begin{equation*}
        f_s(x_{-\beta}(-\frac{t_1^{-1}z}{(1+\abs{t_2}^2)^{\frac{3}{2}}})x_{\alpha+\beta}(-\bar{t_2})k'(t_2)w_\beta g) = f_s(x_{-\beta}(-\frac{t_1^{-1}z}{(1+\abs{t_2}^2)^{\frac{3}{2}}})k'(t_2)w_\beta g).
        \end{equation*}
        We also have the Iwasawa decomposition of $x_{-\beta}(-z)$:
        \begin{equation*}
        x_{-\beta}(-z) = \mtrthr{(\abs{z}^2+1)^{-\frac{1}{2}}}{}{}{}{(\abs{z}^2+1)^{\frac{1}{2}}}{}{}{}{1}x_{\beta}(-\bar{z})k''(z),
        \end{equation*}
        where $k''(z)$ is a smooth function taking values in $K_{\RG_2}$, and $\bar{z}$ is the complex conjugate of $z$. Then the invariant property of $f_s$ yields
        \begin{equation*}
        \begin{aligned}
        &f_s(x_{-\beta}(-\frac{t_1^{-1}z}{(1+\abs{t_2}^2)^{\frac{3}{2}}})k'(t_2)w_\beta g) \\
        =&  \Big\lvert\abs{-\frac{t_1^{-1}z}{(1+\abs{t_2}^2)^{\frac{3}{2}}}}^2+1\Big\lvert_F^{-\frac{3}{2}s}f_s(k''(-\frac{t_1^{-1}z}{(1+\abs{t_2}^2)^{\frac{3}{2}}})k'(t_2)w_\beta j(k)).
        \end{aligned}
        \end{equation*}
        Combining all of the above, we get (\ref{Eq: 06}).
        \end{proof}

        Now we return to the proof of Theorem \ref{Theorem: Absolute Convergence of Integral}.
        \begin{proof}[Proof of Theorem \ref{Theorem: Absolute Convergence of Integral}, the Archimedean Case]
        Using the Iwasawa decomposition for $\SL_3(F)$, we obtain that
        \begin{equation}\label{Eq:03}
        Z(W_v,f_s) = \int_F\int_{A_{\SL_3}}\int_{K_{\SL_3}} W_v(\mtrthr{1}{z}{}{}{1}{}{}{}{1}ak)f_s(\gamma\cdot \mtrthr{1}{z}{}{}{1}{}{}{}{1}ak)\delta_{B_{\SL_3}}^{-1}(a)dkdadz.
        \end{equation}
        Here
        \begin{equation}\label{eq: positive torus}A_{\SL_3} = \bigset{a = \mtrthr{t_1}{}{}{}{t_2}{}{}{}{t_1^{-1}t_2^{-1}}}{t_1,t_2>0},\end{equation} and $K_{\SL_3}$ is the maximal compact subgroup of $\SL_3(F)$ as in \eqref{eq: Iwasawa decomposition}. Thus, by Lemma \ref{lemma: archimedean lemma}, we get
        \begin{equation}\label{Eq: 05}
       \begin{aligned}
            Z(W_v,f_s) &=\int_F\int_{(\BR_+^\times)^2}\int_{K_{\SL_3}}W_v(\mtrthr{t_1t_2}{}{}{}{t_2}{}{}{}{t_1^{-1}t_2^{-2}}k)\psi(z)\abs{t_1^{3s-4}t_2^{9s-6}}_F\cdot\\ &\Big\lvert\abs{t_1^{-1}z}^2+(1+t_2^2)^3\Big\lvert^{-\frac{3}{2}s}_F\cdot f_s(k''(-\frac{-t_1^{-1}z}{(1+t_2^2)^{\frac{3}{2}}})k'(t_2)w_\beta k)dkd^\times t_1d^\times t_2dz.
       \end{aligned}
       \end{equation}
       We change the variable $z \mapsto t_1(1+t_2^2)^{\frac{3}{2}}z$ and rewrite \eqref{Eq: 05} as
       \begin{equation}\label{Eq: 04}
          \begin{aligned}
             Z(W_v,f_s) &=\int_F\int_{(\BR_+^\times)^2}\int_{K_{\SL_3}}W_v(\mtrthr{t_1t_2}{}{}{}{t_2}{}{}{}{t_1^{-1}t_2^{-2}}k)\psi(t_1(1+t_2^2)^{\frac{3}{2}}z)\abs{t_1}_F^{3s-3}\abs{t_2}_F^{9s-6}\\
               \quad&\cdot\Big\lvert\abs{z}^2+1\Big\lvert^{-\frac{3}{2}s}_F\cdot \Big\lvert 1+t_2^2 \Big\lvert_F^{-\frac{9}{2}s+\frac{3}{2}}f_s(k''(-z)k'(t_2)w_\beta k) dkd^\times t_1d^\times t_2dz.
          \end{aligned}
       \end{equation}
       Under the assumption that $\pi$ has a trivial central character, we get
       \begin{equation*}
        W_v(\mtrthr{t_1t_2}{}{}{}{t_2}{}{}{}{t_1^{-1}t_2^{-2}}k) = W_v(\mtrthr{t_1\cdot t_1t_2^3}{}{}{}{t_1t_2^3}{}{}{}{1}k).
       \end{equation*}
      We change the variable $t_2'=t_1t_2^3$ and rewrite (\ref{Eq: 04}) as
      \begin{equation}\label{Eq:07}
        \begin{aligned}
           Z(W_v,f_s)
                      &= \int_F\int_{(\BR_+^\times)^2}\int_{K_{\SL_3}} W_v(\mtrthr{t_1t_2'}{}{}{}{t_2'}{}{}{}{1}k)\psi(t_1(1+t_1^{-\frac{2}{3}}t_2'^{\frac{2}{3}})^{\frac{3}{2}}z)\\
                      &\cdot\Big\lvert\frac{t_1^{3s}t_2'^{3s}}{(t_1^{\frac{2}{3}}+t_2'^{\frac{2}{3}})^{\frac{9}{2}s}}\Big\lvert_F\cdot\abs{t_1^{-2}t_2'^{-2}}_F\cdot\abs{\abs{z}^2+1}_F
          ^{-\frac{3}{2}s}\cdot\abs{t_1^{\frac{2}{3}}+t_2'^{\frac{2}{3}}}^{\frac{3}{2}}_F\\
          &\cdot \frac{1}{3}f_s(k''(-z)k'(t_1^{-\frac{1}{3}}t_2'^{\frac{1}{3}})w_\beta j(k)) dkd^\times t_1d^\times t_2dz.
          \end{aligned}
      \end{equation}
      To save notations, we change $t_2'\mapsto t_2$ in \eqref{Eq:07}. Proposition 2 in \cite[Section 4.3]{J-S} provides a coarse asymptotic expansion of Whittaker functions along the torus: there exists a finite set $X$ of finite functions such that
      \begin{equation}\label{Eq: Course Asymp}
       W_v(\mtrthr{t_1t_2}{}{}{}{t_2}{}{}{}{1}k) = \sum_{\xi\in X} \varphi_\xi(t_1,t_2,k)\xi(t_1,t_2),
      \end{equation}
      where $\varphi_\xi\in S(F^2\times K_{\GL_3})$ are Schwartz functions. After integrating over $K_{\GL_3}$ in (\ref{Eq:07}) and using the asymptotic expansion \eqref{Eq: Course Asymp}, we can see that $Z(W_v,f_s)$ is majorized by a finite linear combination of some integrals of the following type
      \begin{equation*}
       \int_F\int_{(\BR_+^\times)^2} \varphi(t_1,t_2)\xi(t_1,t_2)\cdot\bigg\lvert{\bigg(\frac{t_1^\frac{2}{3}t_2^\frac{2}{3}}{t_1^{\frac{2}{3}}+t_2^{\frac{2}{3}}}\bigg)^{\frac{9}{2}s-\frac{3}{2}}}\bigg\rvert_F \cdot \abs{t_1^{-1}t_2^{-1}}_F\cdot\big\lvert{(\abs{z}^2+1)^{-\frac{3}{2}s}}\big\rvert_F d^\times t_1 d^\times t_2 dz,
      \end{equation*}
      where $\varphi$ is a nonnegative Schwartz function on $F^2$.

      Let us first assume that Re$(s)> \frac{1}{3}$, then clearly the $dz$-integral converges. By Cauchy inequality, we have
      \begin{equation*}
       \bigg\lvert{\bigg(\frac{t_1^\frac{2}{3}t_2^\frac{2}{3}}{t_1^{\frac{2}{3}}+t_2^{\frac{2}{3}}}\bigg)^{\frac{9}{2}s-\frac{3}{2}}}\bigg\rvert_F\leq \abs{t_1^\frac{1}{3}t_2^\frac{1}{3}\big}_F^{\frac{9}{2}\text{Re}(s)-\frac{3}{2}}.
      \end{equation*}
      Thus, $Z(W_v,f_s)$ is majorized by a finite linear combination of integrals of the following type:
      \begin{equation*}
       \int_{(\BR_+^\times)^2}\varphi(t_1,t_2)\xi(t_1,t_2)\cdot\abs{t_1t_2}^{\frac{3}{2}\text{Re}(s)-\frac{3}{2}}_F d^\times t_1 d^\times t_2.
      \end{equation*}
      Therefore, when Re$(s)$ is sufficiently large, $Z(W_v, f_s)$ converges absolutely.
      \end{proof}

%%%%%%%%%%%%%%%%%%%%%%%%%%%%%%%%%%%%%%%%%%%%%%%%%%%%%%%%%%%%%%%%%%%%%%%%%%%%%%%%%%%%%%%%%%%%%%%%%%%%%%%%%%%%%%%%%%%%%%%%%%%%%%%%%%%%%%%%%%%

    \subsection{Some Quick Remarks on the Non-archimedean Case}\label{subsection: non-archi remark}

    Let us conclude this Section by making some remarks in the non-archimedean case. In fact, we can prove Theorem \ref{Theorem: Absolute Convergence of Integral} in the non-archimedean case via the same method. However, we can move one step further and prove convergence of the local integral and its meromorphic continuation together in the non-archimedean case by using the $K$-finiteness conditions of $W_v$ and $f_s$. Here we reformulate the theorem as below, whose proof can be found in my thesis (see \cite[Theorem 3.2.1]{TianThesis})
    \begin{thm}\label{Theorem: Meromorphic Continuation p-adic}
          Suppose that $F$ is non-archimedean. The local integral $Z(W_v,f_s)$ converges absolutely when $\Re(s)$ is sufficiently large. They are rational functions of $q^{-3s}$, where $q$ is the cardinality of the residue field. Hence they automatically enjoy a meromorphic continuation to the whole complex plane. When $\pi$ is a principal series (not necessarily unramified): $$\pi = \Ind_{B_{\GL_3}}^{\GL_3(F)} \abs{\quad}^{u_1}\chi_1 \otimes \abs{\quad}^{u_2}\chi_2\otimes \abs{\quad}^{u_3}\chi_3,$$
          where $u_1,u_2,u_3\in \BC$ and each $\chi_j$ $(j=1,2,3)$ is a character of the group of units of the ring of the integers of $F$, the local integrals \eqref{Def: local archi integral} are also meromorphic in $u_1,u_2,u_3$.
    \end{thm}
    I also proved the following theorem in my thesis in detail (\cite[Theorem 4.1.1]{TianThesis}):
    \begin{thm}\label{Theorem: Non-archi local int const}
        When $F$ is non-archimedean, there exists a Whittaker function $W_v\in\mathcal{W}(\pi,\psi)$ and a function $f_s\in V_{\rho_s}$ such that $Z(W_v,f_s)=1$ for every $s\in \BC$.
    \end{thm}
    The above two theorems can be viewed as special cases of the recent work of Hundley and Zhang (see \cite[Lemma 3.1]{HundleyZhangAdjointGL3}), yet unfortunately they did not provide enough details in their paper.

    Now we fix one non-trivial additive character $\psi$ of $F$. By Pontryagin Duality, any non-trivial additive character of $F$ is of the form
    \begin{equation*}
        \psi_c(x) = \psi(cx) \quad\text{for some } c\in F^\times.
    \end{equation*}
    To emphasize the dependence of Whittaker functions on the additive character, we write $W_v^\psi$ for a Whittaker function in $\mathcal{W}(\pi,\psi)$. Note that if $W_v^\psi\in\mathcal{W}(\pi,\psi)$, then the function $W_v^{\psi_c}$ defined by
    \begin{equation*}
        W_v^{\psi_c}(g) = W_v^{\psi}(\mtrthr{c}{}{}{}{1}{}{}{}{c^{-1}}g)
    \end{equation*}
    lives in $\mathcal{W}(\pi,\psi_c)$. It is easy to check that
    \begin{equation}\label{Eq: 26}
       Z(W_v^{\psi_c},f_s) = \abs{c}_F^{-3s+3} Z(W_v^{\psi},f_s).
    \end{equation}
    When we vary the Whittaker models of $\pi$ and consider all such local integrals $Z(W_v,f_s)$, all of these $Z(W_v,f_s)$ span a subspace of $\BC(q^{-3s})$ which is closed under multiplication by $q^{-3s}$ and $q^{3s}$. It is a fractional ideal containing 1 by Theorem \ref{Theorem: Non-archi local int const}. Thus, it is generated by a function $\frac{1}{P(q^{-3s})}$ for some polynomial $P$. We expect that $\frac{1}{P(q^{-3s})}$ is essentially the adjoint L-function for $\pi$, in the sense of Langlands. We will consider this problem in a future work.

%%%%%%%%%%%%%%%%%%%%%%%%%%%%%%%%%%%%%%%%%%%%%%%%%%%%%%%%%%%%%%%%%%%%%%%%%%%%%%%%%%%%%%%%%%%%%%%%%%%%%%%%%%%%%%%%%%%%%

%%%%%%%%%%%%%%%%%%%%%%%%%%%%%%%%%%%%%%%%%%%%%%%%%%%%%%%%%%%%%%%%%%%%%%%%%%%%%%%%

%%%%%%%%%%%%%%%%%%%%%%%%%%%%%%%%%%%%%%%%%%%%%%%%%%%%%%%%%%%%%%%%%%%%%%%%%%%%%%%%
% reconstruction.tex:
%%%%%%%%%%%%%%%%%%%%%%%%%%%%%%%%%%%%%%%%%%%%%%%%%%%%%%%%%%%%%%%%%%%%%%%%%%%%%%%%

\section{Asymptotic Expansion of the Whittaker Functions (Archimedean Cases)} \label{section: aysmp}

%%%%%%%%%%%%%%%%%%%%%%%%%%%%%%%%%%%%%%%%%%%%%%%%%%%%%%%%%%%%%%%%%%%%%%%%%%%%%%%%

    In this Section, we assume that $F$ is either $\BR$ or $\BC$. Because the local representation $\pi$ satisfies the moderate growth property, for any $a$ in the torus of $\GL_3(F)$, the Whittaker function $W_v(a)$ behave like a Schwartz function when $a$ approaches to infinity. Thus, the asymptotic behaviour of $W_v(a)$ near $0$ should be responsible for the poles of the local integrals $Z(W_v,f_s)$.

    Recall that in Section \ref{section: convergence}, we use the coarse asymptotic expansion \eqref{Eq: Course Asymp} to prove the Convergence Theorem (Theorem \ref{Theorem: Absolute Convergence of Integral}). The Schwartz functions $\varphi_\xi$ in (\ref{Eq: Course Asymp}) depend on the vector $v\in V_\pi$. Unfortunately, when we vary $v\in V_\pi$, we can not see how all these $\varphi_\xi$ change accordingly. In this Section, we will follow Soudry's method (see \cite{Sou}) closely and obtain an asymptotic expansion of $W_v(a)$ near 0. We will see that each term in the desired asymptotic expansion in this Section is continuous in $v$.

%%%%%%%%%%%%%%%%%%%%%%%%%%%%%%%%%%%%%%%%%%%%%%%%%%%%%%%%%%%%%%%%%%%%%%%%%%%%%%%%%%%%%%%%%%%%%%%%%%%%%%%%%%%%%%%%%%%%%%%%%%%%%%%%%%%%%%%%%%%%%%%
    \subsection{First Step}
    We first fix a norm on $\GL_3(F)$: for any real or complex matrix $g$, define its Euclidean norm
    \begin{equation*}
       \norm{g}_e = \text{Tr}(g\cdot\bar{g}^{t})^{\frac{1}{2}} = (\sum \abs{g_{ij}}^2)^{\frac{1}{2}}.
    \end{equation*}
    Moreover, when $g\in\GL_3(F)$, we define its Harish-Chandra norm by
    \begin{equation*}
       \norm{g}_{H} = \norm{g}_e^2+\norm{g^{-1}}_e^2.
    \end{equation*}
    In particular, when $g=\mtrthr{a_1}{}{}{}{a_2}{}{}{}{a_3}$, $\norm{g}_H = \sum_{i=1}^3 \abs{a_i}^2+\abs{a_i}^{-2}$. It is easy to see that
    \begin{equation*}
        \norm{\mtrthr{a_1}{}{}{}{a_2}{}{}{}{a_3}\mtrthr{b_1}{}{}{}{b_2}{}{}{}{b_3}}_H\leq \norm{\mtrthr{a_1}{}{}{}{a_2}{}{}{}{a_3}}_H\cdot\norm{\mtrthr{b_1}{}{}{}{b_2}{}{}{}{b_3}}_H.
    \end{equation*}
    In the following, we will always use Harish-Chandra norm and drop subscript $H$.

    As in the Introduction, we assume that $\pi$ is an irreducible generic representation of $\GL_3(F)$. Moreover, if we consider a principal series $\pi$ as in \eqref{eq: principal series}, then we also write $\pi = \pi_u, W_v = W_{v,u}$ to emphasize their dependence on the complex parameters $u = (u_1,u_2,u_3) \in \BC^3$.

    By the continuity of the Whittaker functional, there exist a $\mu>0$ and a continuous seminorm $q$ on $V_\pi$ such that
    \begin{equation}\label{Eq: 08}
       \abs{W_v(g)}\leq \norm{g}^{\mu}\cdot q(v),
    \end{equation}
    for any $v\in V_\pi$. If a compact set $\Omega\in \BC^3$ is given, then a similar inequality
    \begin{equation}\label{Eq: 07}
        \abs{W_{v,u}(g)}\leq \norm{g}^{\mu}\cdot q(v)
    \end{equation}
    also holds for all $v\in V_{\pi_u}$, $u\in \Omega$. The constant $\mu$ and seminorm $q$ in (\ref{Eq: 07}) only depends on $\Omega$.

    Put
    \begin{equation*}
       \begin{aligned}
       H_0 &= -\mtrthr{1}{}{}{}{1}{}{}{}{1}, &H_1&= -\mtrthr{1}{}{}{}{0}{}{}{}{0}, &H_2&= -\mtrthr{1}{}{}{}{1}{}{}{}{0},\\
       X_1 &= \mtrthr{0}{1}{}{}{0}{}{}{}{0}, &X_2&= \mtrthr{0}{}{}{}{0}{1}{}{}{0}, &X_3&= \mtrthr{0}{}{1}{}{0}{}{}{}{0},
       \end{aligned}
    \end{equation*}
    and set $\alpha_1 = \eps_1-\eps_2$, $\alpha_2 = \eps_2-\eps_3$. Then
    \begin{equation*}
       \alpha_1(H_0) = \alpha_2(H_0) = 0, \alpha_1(H_1) = \alpha_2(H_2)=-1, \alpha_1(H_2) = \alpha_2(H_1) = 0.
    \end{equation*}

    For $m=1,2,$ let $P_m$ be the standard parabolic subgroup of $\GL_3(F)$ of type $(m,3-m)$ with the Levi decomposition $P_m=M_mU_m$. Denote by $V_0$ ($V_{0,u}$ resp.) the underlying $(\frak{gl}_3, K_{\GL_3})$-module of $V_\pi$ ($V_{\pi_u}$ resp.). Set $K_m = M_m\cap K_{\GL_3}$. By \cite[Section 4.3.1]{Wal1}, for any positive integer $k$, $V_0/\Lie(U_m)^kV_0$ is a finitely generated admissible $(\Lie(M_m), K_m)$-module, and it decomposes into finitely many generalized $H_m$-eigenspaces
    \begin{equation*}
       V_0/\text{Lie}(U_m)^kV_0 = \bigoplus_{\xi} \big(V_0/\text{Lie}(U_m)^kV_0\big)_\xi,
    \end{equation*}
    where
    \begin{equation*}
       (V_0/\text{Lie}(U_m)^kV_0\big)_\xi := \set{\bar{v}\in V_0/\text{Lie}(U_m)^kV_0}{(\pi(H_m)-\xi)^d\bar{v} = 0 \text{ for some }d>0}.
    \end{equation*}
    Denote by $E_k^{(m)}$ ($E_{k,u}^{(m)}$ resp.) the finite set of generalized eigenvalues of $H_m$ on \\$V_0/\text{Lie}(U_m)^kV_0$ ($V_{0,u}/\text{Lie}(U_m)^kV_{0,u}$ resp.). In particular, $E_1^{(m)}$ ($E_{1,u}^{(m)}$ resp.) is the finite set of generalized eigenvalues of $H_m$ on $V_0/\text{Lie}(U_m)V_0$ ($V_{0,u}/\text{Lie}(U_m)V_{0,u}$ resp.). Define
    \begin{equation*}
         E^{(m)} := \bigcup_{k=1}^{+\infty} E_k^{(m)},\qquad E_u^{(m)} := \bigcup_{k=1}^{+\infty} E_{k,u}^{(m)}.
    \end{equation*}
    By \cite[Section 4.4.2]{Wal1},
    \begin{equation*}
       E^{(m)} \sbst \set{\xi-n}{\xi\in E_1^{(m)}, n\in\BN},\quad  E_u^{(m)} \sbst \set{\xi-n}{\xi\in E_{1,u}^{(m)}, n\in\BN}.
    \end{equation*}
    Moreover, by \cite[Section 12.4.6 and 12.4.7]{Wal2}, $E_{1,u}^{(m)}$ consists of polynomials of $u_1,u_2,u_3$ of degree 1.

    In \cite[Section 15.2.4]{Wal2}, Wallach proves the following lemma (we only state the lemma in the case of $\frak{gl}_3$).
    \begin{lemma}\label{Lemma: Wallach 15.2.4}
       For each $m\in \{1,2\}$, and $k_m\in\BN$, there exists an integer $N(m)\in\BN$, a finite set $\{e_i^{(m)}\}_{i=1}^{N(m)}\sbst \textit{U}(\mathfrak{gl}_3(\BC))$, a finite set $\{D_{r,i}^{(m)}\}\sbst\textit{U}(\frak{gl}_3(\BC))$, where $r$ indexes the basis of monomials $X_r^{(m)}$ in $\Lie(U_m)^{k_m}$, (for example, if $m=1$, $X_r^{(1)}$ is of the form $X_3^{l_3}X_2^{l_2}$) and $i = 1,2,\cdots, N(m)$, such that for any $(\frak{gl}_3, K_{\GL_3})$-module $V_0$, there exists an $N(m)\times N(m)$ matrix $B^{(m,k_m)} = (B_{ij}^{(m,k_m)})$ depending on $V_0$ such that
       \begin{enumerate}
         \item $e_1^{(m)} = 1$,
         \item For any $v\in V_0$,
               \begin{equation}\label{Eq: Equation from Wallach 15.2.4}
                   \pi(H_m)\pi(e_i^{(m)})v = \sum_j B_{ij}^{(m,k_m)}\pi(e_j^{(m)})v+\sum_r\pi(X_r^{(m)})\pi(D_{r,i}^{(m)})v.
               \end{equation}
       \end{enumerate}
    \end{lemma}
    If we take $V_0$ to be the $(\mathfrak{gl}_3, K_{\GL_3})$-module of $\pi_u$ in the above lemma, then we write
    \begin{equation*}
        B^{(m,k_m)} =  B^{(m,k_m)}_u = (B_{ij,u}^{(m,k_m)})
    \end{equation*}
    to emphasize the dependence on the complex parameter $u$. If we fix one $k_m$, we also omit the superscript $k_m$ and write $B^{(m)} (B^{(m)}_u$ resp.) for $B^{(m,k_m)}$ ($B^{(m,k_m)}_u$ resp.). Thus, we can conclude from Lemma \ref{Lemma: Wallach 15.2.4} that all of the generalized eigenvalues of $B^{(m,k_m)}_u$ are contained in $E_{k_m,u}^{(m)}$, and hence they are polynomials of $u$ of degree 1. Moreover by \cite[Section 12.4.7]{Wal2}, each entry of $B^{(m,k_m)}_u$ is rational in $u_1,u_2,u_3$.

    We label all elements in $E^{(m)} = \set{\xi_j^{(m)}}{j=1,2,\cdots}$ in the descending order, i.e.
    \begin{equation*}
       \text{Re}(\xi_1^{(m)})\geq \text{Re}(\xi_2^{(m)})\geq\cdots.
    \end{equation*}
    We also write $\xi_{j,u}^{(m)}$ when we emphasize their dependence on parameters. In this Section, we fix two numbers $\xi_m$ for $m=1,2$. Later, when we prove Theorem \ref{Theorem: Meromorphic Continuation}, we will choose $\xi_1$, $\xi_2$ as negative as we want to cover enough generalized eigenvalues.

    \begin{thm}\label{Theorem: Asymptotic Expansion second variable}
       Fix $\xi_2\in \BC$. Suppose $V_0 = V_{0,u}$, where $u$ runs in a fixed closed ball $\Omega = B(u_0,r_0)\in \BC^3$ ($u_0$ is the center) with a radius $r_0>0$.  Let $k_2$ be a sufficiently large positive integer such that
       \begin{equation*}
             -\Re(\xi_2)-k_2+2\mu<-1.
       \end{equation*}
       There is a finite set $C_u^{(2)}\sbst\bigcup_{j=1}^{k_2}E_{j,u}^{(2)}$, and a finite set $\mathcal{L}$ of nonnegative integers, a finite subset $\mathcal{P}_u\sbst \BC(u)[t]$ and a finite subset $\mathcal{D}\sbst U(\frak{gl}_3(\BC))$ such that for all $v\in V_{0,u}$, $x_2 \geq 0$, $W_{v,u}(e^{x_1H_1+x_2H_2})$ is a finite linear combination of terms of the following types:
       \begin{enumerate}
         \item  $e^{x_2\xi_u}p_u(x_2)W_{\pi(e)v,u}(e^{x_1H_1}),$
         \item  $e^{x_2\xi_u}h_u(x_2)\int_{0}^{x_2}e^{(-\xi_u-k_2)t}t^lW_{\pi(D)v,u}(e^{x_1H_1+tH_2})dt := h_u(x_2)\tilde{\phi}_{l,e,D}(x_1,x_2,v,u),$
         \item  $e^{x_2\xi_u}h_u(x_2)\int_0^{+\infty}e^{(-\xi_u-k_2)t}t^lW_{\pi(D)v,u}(e^{x_1H_1+tH_2})dt := e^{x_2\xi_u}h_u(x_2)\phi_{l,e,D}(x_1,x_2,v,u) ,$
         \item  $e^{x_2\xi_u}h_u(x_2)\int_{x_2}^{+\infty}e^{(-\xi_u-k_2)t}t^lW_{\pi(D)v,u}(e^{x_1H_1+tH_2})dt := h_u(x_2)\tilde{\phi}_{l,e,D}(x_1,x_2,v,u),$

       \end{enumerate}
       where $p_u, h_u\in\mathcal{P}_u$, $e,D \in \mathcal{D}$, and $\xi_u\in C_u^{(2)}$ are polynomials of $u$ of degree 1. For each $\xi_u\in C_u^{(2)}$, we decompose $\Omega = \Omega_{\xi_u,1}\bigcup\Omega_{\xi_u, 2}$ where
       \begin{equation*}
            \begin{aligned}
               \Omega_{\xi_u,1} := \set{u\in \Omega}{-k_2-\Re(\xi_u)+2\mu \geq -\frac{1}{2}},\\
               \Omega_{\xi_u,2} := \set{u\in \Omega}{-k_2-\Re(\xi_u)+2\mu \leq -\frac{1}{3}}.\\
            \end{aligned}
       \end{equation*}

       If $u\in \Omega_{\xi_u,1}$, then we only need terms of type (1) and type (2). The function $\phi_{l,e,D}(x_1,x_2,v,u)$ defined in type (2) is holomorphic in $u$ and uniformly continuous in $x_1,x_2,v$ when $u$ runs in $\Omega_{\xi_u,1}$. It satisfies the following estimate
       \begin{equation}\label{Eq: Estimate of First Asymp Exp 2nd term 1}
           \begin{aligned}
           &\big\lvert\tilde{\phi}_{l,e,D}(x_1,x_2,v,u)\big\rvert \leq e^{\Re(\xi_2)x_2}\cdot \tilde{\delta}_{l}(x_2)\cdot\norm{e^{x_1H_1}}^\mu \cdot q'(v)
           \end{aligned}
        \end{equation}
       where $q'$ is a continuous seminorm and $\tilde{\delta}_l(t) = \frac{t^{l+1}}{l+1}$.

       If $u\in \Omega_{\xi_u,2}$, we only need terms of type (1), (3) and (4). The functions \\
       $\phi_{l,e,D}(x_1,x_2,v,u)$ and $\tilde{\phi}_{l,e,D}(x_1,x_2,v,u)$ are also holomorphic in $u\in \Omega_{\xi_u,2}$ and uniformly continuous in $x_1,x_2,v$ when $u$ runs in $\Omega_{\xi_u,2}$. They satisfy the following estimates:
       \begin{equation}\label{Eq: Estimate of First Asymp Exp 3rd term}
           \big\lvert \phi_{l,e,D}(x_1,x_2,v,u)\big\rvert \leq \delta_l(x_2)\cdot\norm{e^{x_1H_1}}^\mu\cdot q'(v);
       \end{equation}
       \begin{equation}\label{Eq: Estimate of First Asymp Exp 4th term 1}
           \big\lvert \tilde{\phi}_{l,e,D}(x_1,x_2,v,u)\big\lvert
           \leq e^{\Re(\xi_2)x_2}\tilde{\delta}_l(x_2)\cdot\norm{e^{x_1H_1}}^\mu\cdot q'(v),
       \end{equation}
       where $q'$ is a continuous seminorm, $\delta_l(t)$ is the constant function
       \begin{equation*}
            \delta_l(t) = \int_{0}^{+\infty}e^{-\frac{1}{3}t}t^ldt
       \end{equation*}
       and $\tilde{\delta}_l(t)$ is a polynomial of $t$ of degree $l$ with constant coefficients coming from integration by parts.

       Same expansions and estimates also hold if we drop the subscript $u$ and consider the underlying $(\frak{gl}_3, K_{\GL_3})$-module for all irreducible generic Casselman-Wallach representations.
    \end{thm}

    \begin{proof}[Proof of Theorem \ref{Theorem: Asymptotic Expansion second variable}]
        We only look at the case when $\pi = \pi_u$ and keep track on the dependence on the parameters $u$. A word by word repetition also works for all irreducible admissible generic Casselman-Wallach representations $\pi$ when we drop the subscript $u$.

        For $k_2$, we choose $N(2)\in\BN$, $\{e_i^{(2)}\}_{i=1}^{N(2)}$, $\{D_{r,i}^{(2)}\}$, $B_u^{(2)}$ as in Lemma \ref{Lemma: Wallach 15.2.4}. Put
        \begin{equation*}
           \vec{F}_u(x_1,t,v) := \left(
                                \begin{array}{c}
                                  W_{\pi(e_1^{(2)})v,u}(e^{x_1H_1+tH_2}) \\
                                  \cdots \\
                                  W_{\pi(e_{N(2)}^{(2)})v,u}(e^{x_1H_1+tH_2}) \\
                                \end{array}
                              \right),
        \end{equation*}
        and
        \begin{equation*}
           \vec{G}_u(x_1,t,v) := \left(
                                \begin{array}{c}
                                  \sum_r W_{\pi(X_r^{(2)})\pi(D_{r,1}^{(2)})v,u}(e^{x_1H_1+tH_2}) \\
                                  \cdots \\
                                  \sum_r W_{\pi(X_r^{(2)})\pi(D_{r,N(2)}^{(2)})v,u}(e^{x_1H_1+tH_2}) \\
                                \end{array}
                              \right).
        \end{equation*}
       According to (\ref{Eq: Equation from Wallach 15.2.4}), it is clear that
        \begin{equation*}
           \begin{aligned}
           &\frac{d}{dt}W_{\pi(e_i^{(2)})v,u}(e^{x_1H_1+tH_2}) \\
           = &\sum_{j}B_{ij,u}^{(2)}\cdot W_{\pi(e_j^{(2)})v,u}(e^{x_1H_1+tH_2})+\sum_r W_{\pi(X_r^{(2)})\pi(D_{r,i}^{(2)})v,u}(e^{x_1H_1+tH_2}),
           \end{aligned}
        \end{equation*}
        i.e.
        \begin{equation}\label{Eq: Whittaker differential eq}
           \frac{d}{dt}\vec{F}_u(x_1,t,v) = B^{(2)}_u\vec{F}_u(x_1,t,v)+\vec{G}_u(x_1,t,v).
        \end{equation}

        Note that Lie$(U_2)$ is spanned by $X_3$ and $X_2$, so every $X_r^{(2)}$ is of the form $X_3^{l_3}X_2^{l_2}$ with $l_3+l_2=k_2$.
        If $l_3=0$, in other words $l_2=k_2$ and $X_r^{(2)} = X_2^{k_2}$, we define $\tilde{D}_{r,i}^{(2)}={D}_{r,i}^{(2)}$. Then we have
        \begin{equation*}
           W_{\pi(X_r^{(2)})\pi(D_{r,i}^{(2)})v,u}(e^{x_1H_1+tH_2}) = C_\psi e^{-k_2t}W_{\pi(\tilde{D}_{r,i}^{(2)})v,u}(e^{x_1H_1+tH_2})
        \end{equation*}
        for some non-zero $C_\psi\in\BC$, only depending on $\psi$ and $F$. By rescaling $\tilde{D}_{r,i}^{(2)}$, we rewrite the above equation as
        \begin{equation}\label{Eq: 190}
           W_{\pi(X_r^{(2)})\pi(D_{r,i}^{(2)})v,u}(e^{x_1H_1+tH_2}) = e^{-k_2t}W_{\pi(\tilde{D}_{r,i}^{(2)})v,u}(e^{x_1H_1+tH_2}).
        \end{equation}
        If $l_3>0$, we define $\tilde{D}_{r,i}^{(2)}=0$. Then it is easy to check that
        \begin{equation*}
           \begin{aligned}
           W_{\pi(X_r^{(2)})\pi(D_{r,i}^{(2)})v,u}(e^{x_1H_1+tH_2}) = 0.
           \end{aligned}
        \end{equation*}
        Hence (\ref{Eq: 190}) also holds in this case.

        The solution to the differential equation (\ref{Eq: Whittaker differential eq}) is
        \begin{equation}\label{Eq: 09}
           \vec{F}_u(x_1,x_2,v) = e^{x_2B^{(2)}_u}\vec{F}_u(x_1,0,v)+\int_0^{x_2}e^{(x_2-t)B^{(2)}_u}\vec{G}_u(x_1,t,v)dt
        \end{equation}

        We compare the first coordinates of both sides in (\ref{Eq: 09}). The first coordinate on the LHS of (\ref{Eq: 09}) is $W_v(e^{x_1H_1+x_2H_2})$. Since each entry of $B^{(2)}_u$ is rational in $u$, each entry of the matrix $e^{x_2B^{(2)}_u}$ is a linear combination
        \begin{equation*}
             e^{x_2\xi_u}p_{u}(x_2),
        \end{equation*}
        where $p_u(x_2)$ belongs to a finite subset of $\BC(u)[x_2]$, and $\xi_u$ lies in a finite subset of $\bigcup_{i=1}^{k_2}E_{i,u}^{(2)}$.
        Thus, the first coordinate of the first term on the right hand side of (\ref{Eq: 09}) is a finite linear combination of
        \begin{equation}\label{Eq: First Asymp Exp 1st term}
           e^{x_2\xi_u}p_{u}(x_2)W_{\pi(e_i)}(e^{x_1H_1}),
        \end{equation}
        where $\xi_u$ belongs to a finite subset of $\bigcup_{i=1}^{k_2}E_{i,u}^{(2)}$ consisting of polynomials of degree 1, and $p_{u}(t)$ belongs to a finite subset of $\BC(u)[t]$.

        The first coordinate of the second term on the RHS of (\ref{Eq: 09}) is a finite linear combination of
        \begin{equation*}
           \int_0^{x_2} e^{(x_2-t)\xi_u} p_{u}(x_2-t)\cdot e^{-k_2t}W_{\pi(\tilde{D}_{r,i}^{(2)})v,u}(e^{x_1H_1+tH_2})dt,
        \end{equation*}
        where $\xi_u$, $p_u(t)$ satisfy the same conditions as above. After expanding the polynomials $p_{u}(x_2-t)$ into monomials, we can see that the first coordinate of the second term on the RHS of (\ref{Eq: 09}) is a finite linear combination of
        \begin{equation}\label{Eq: First Asymp Exp 2nd term}
           e^{x_2\xi_u}h_u(x_2)\int_0^{x_2}e^{(-\xi_u-k_2)t}t^lW_{\pi(D)v,u}(e^{x_1H_1+tH_2})dt,
        \end{equation}
        where $h_u(t)$ belongs to a finite subset of $\BC(u)[t]$, $\xi_u$ lies in a finite subset of $\bigcup_{i=1}^{k_2}E_{i,u}^{(2)}$, $l$ runs in a finite set $\mathcal{L}$ of nonnegative integers, and $D$ belongs to a finite set $\{\tilde{D}_{r,i}^{(2)}\}$. We also note that $\mathcal{L}$ only depends on the size of $B_u^{(2)}$, which is independent on the parameter $u$.

        Since the continuous function $-\Re(\xi_u)$ is bounded on the compact set $\Omega$, we may assume that $-\Re(\xi_u)<C$ for some constant $C$. Then the integral in (\ref{Eq: First Asymp Exp 2nd term}) admits the following estimate
        \begin{equation}\label{Eq: Estimate of First Asymp Exp 2nd term}
           \begin{aligned}
           &\big\lvert\int_0^{x_2}e^{(-\xi_u-k_2)t}t^lW_{\pi(D)v}(e^{x_1H_1+tH_2})dt\big\rvert\\
           \leq &\int_0^{x_2}e^{(-\Re(\xi_u)-k_2)t}t^l\norm{e^{x_1H_1+tH_2}}^{\mu}dt\cdot q(\pi(D)v)\\
           \leq &\int_0^{x_2}e^{(C-k_2)t}t^l\norm{e^{x_1H_1}}^\mu (2e^{2t}+2e^{-2t}+2)^{\mu}dt\cdot q(\pi(D)v)\\
           \leq &6^\mu \int_0^{x_2}e^{(C-k_2+2\mu)t}t^l\norm{e^{x_1H_1}}^\mu dt\cdot q(\pi(D)v)\\
           =& \int_0^{x_2}e^{(C-k_2+2\mu)t}t^ldt\cdot\norm{e^{x_1H_1}}^\mu \cdot q'(v)\\
           \end{aligned}
        \end{equation}
        for some continuous seminorm $q'$. Thus as a function of $x_1,x_2,u,v$, the integral
        \begin{equation*}
        \int_0^{x_2}e^{(-\xi_u-k_2)t}t^lW_{\pi(D)v,u}(e^{x_1H_1+tH_2})dt
        \end{equation*}
        is holomorphic in $u\in\Omega$ and continuous in $x_1,x_2,v$ (also uniformly continuous when $u$ runs in $\Omega$).

        If $u\in \Omega_{\xi_u,1}$, we can use a similar estimate and get
        \begin{equation}\label{Eq: 114}
           \begin{aligned}
           &\big\lvert\int_0^{x_2}e^{(-\xi_u-k_2)t}t^lW_{\pi(D)v}(e^{x_1H_1+tH_2})dt\big\rvert\\
           \leq &6^\mu \int_0^{x_2}e^{(-\Re(\xi_u)-k_2+2\mu)t}t^l\norm{e^{x_1H_1}}^\mu dt\cdot q(\pi(D)v)\\
           =& \int_0^{x_2}e^{(-\Re(\xi_u)-k_2+2\mu+\frac{1}{2})t}\cdot e^{-\frac{1}{2}t}t^ldt\cdot\norm{e^{x_1H_1}}^\mu \cdot q'(v)\\
           \leq &\int_0^{x_2}e^{(-\Re(\xi_u)-k_2+2\mu+\frac{1}{2})x_2}\cdot t^ldt\cdot\norm{e^{x_1H_1}}^\mu \cdot q'(v)\\
           = & e^{(-\Re(\xi_u)-k_2+2\mu+\frac{1}{2})x_2}\cdot \tilde{\delta}_{l}(x_2)\cdot\norm{e^{x_1H_1}}^\mu \cdot q'(v)\\
           \leq & e^{(\Re(\xi_2)-\Re(\xi_u))x_2}\cdot \tilde{\delta}_{l}(x_2)\cdot\norm{e^{x_1H_1}}^\mu \cdot q'(v)
           \end{aligned}
        \end{equation}
        for some continuous seminorm $q'$ and $\tilde{\delta}_l(t) = \frac{t^{l+1}}{l+1}$.

        We now claim that the integral
        \begin{equation}\label{Eq: First Asymp Exp 3rd term}
           \int_0^{+\infty}e^{(-\xi_u-k_2)t}t^lW_{\pi(D)v}(e^{x_1H_1+tH_2})dt
        \end{equation}
        converges absolutely when $u\in \Omega_{\xi_u,2}$. Indeed, by imitating the estimate (\ref{Eq: 114}), we can show that the integrand in (\ref{Eq: First Asymp Exp 3rd term}) is dominated by
        \begin{equation*}
            e^{(-\Re(\xi_u)-k_2)t}t^l\cdot\norm{e^{x_1H_1}}\cdot 6^\mu e^{2\mu t}q_1(v)
        \end{equation*}
        for some continuous seminorm $q_1$. Thus,
        \begin{equation}\label{Eq: 115}
           \begin{aligned}
           \big\lvert \int_0^{+\infty}e^{(-\xi_u-k_2)t}t^lW_{\pi(D)v}(e^{x_1H_1+tH_2})dt \big\rvert &\leq \int_{0}^{+\infty}e^{(-\Re(\xi_u)-k_2+2\mu)t}t^ldt\cdot\norm{e^{x_1H_1}}^\mu\cdot q'(v)\\
           &\leq \delta_l(x_2)\cdot\norm{e^{x_1H_1}}^\mu\cdot q'(v),
           \end{aligned}
        \end{equation}
        where $q'$ is a continuous seminorm and $\delta_l(t)$ is the constant function
        \begin{equation*}
             \int_{0}^{+\infty}e^{-\frac{1}{3}t}t^ldt.
        \end{equation*}
        Hence our claim follows. Moreover, the estimate (\ref{Eq: 115}) also shows that as a function of $x_1, u, v$, the integral (\ref{Eq: First Asymp Exp 3rd term}) is holomorphic in $u$ and uniformly continuous in $x_1, v$ when $u$ runs in $\Omega_{\xi_u,2}$.

        For $u\in \Omega_{\xi_u,2}$, now we can write
        \begin{equation*}
            \begin{aligned}
            &\int_0^{x_2}e^{(-\xi_u-k_2)t}t^lW_{\pi(D)v}(e^{x_1H_1+tH_2})dt \\ =&\int_0^{+\infty}e^{(-\xi_u-k_2)t}t^lW_{\pi(D)v}(e^{x_1H_1+tH_2})dt-\int_{x_2}^{+\infty}e^{(-\xi_u-k_2)t}t^lW_{\pi(D)v}(e^{x_1H_1+tH_2})dt.
            \end{aligned}
        \end{equation*}
        By (\ref{Eq: 115}), each summand on the RHS of the above equation converges absolutely and defines a holomorphic function in $u$ which is also uniformly continuous in $x_1,x_2,v$ when $u$ runs in $\Omega_{\xi_u,2}$. By using the same trick, we can check that the second summand above admits the following estimate
        \begin{equation}\label{Eq: Estimate of First Asymp Exp 4th term}
           \begin{aligned}
           \big\lvert \int_{x_2}^{+\infty}e^{(-\xi_u-k_2)t}t^lW_{\pi(D)v}(e^{x_1H_1+tH_2})dt\big\lvert
           \leq &\int_{x_2}^{+\infty}e^{(-\Re(\xi_u)-k_2+2\mu)t}t^ldt\cdot\norm{e^{x_1H_1}}^\mu\cdot q'(v)\\
           \leq &e^{(-\Re(\xi_u)-k_2+2\mu)x_2}\tilde{\delta}_l(x_2)\cdot\norm{e^{x_1H_1}}^\mu\cdot q'(v)
           \end{aligned}
        \end{equation}
        where $\tilde{\delta}_l(t)$ is a polynomial of $t$ of degree $l$ coming from integration by parts. Thus,
        \begin{equation}\label{Eq: Estimate of First Asymp Exp 4th term 2}
           \begin{aligned}
           \big\lvert e^{x_2\xi_u}\int_{x_2}^{+\infty}e^{(-\xi_u-k_2)t}t^lW_{\pi(D)v}(e^{x_1H_1+tH_2})dt\big\lvert
           \leq &e^{(-k_2+2\mu)x_2}\tilde{\delta}_l(x_2)\cdot\norm{e^{x_1H_1}}^\mu\cdot q'(v)\\
           \leq &e^{\Re(\xi_2)x_2}\tilde{\delta}_l(x_2)\cdot\norm{e^{x_1H_1}}^\mu\cdot q'(v)
           \end{aligned}
        \end{equation}
        \end{proof}

        We can also fix $\xi_1$ and prove a similar theorem by switching the role of $x_1$ and $x_2$. We only state the result. The proof is exactly the same as that of Theorem \ref{Theorem: Asymptotic Expansion second variable}.
        \begin{thm}\label{Theorem: Asymptotic Expansion first variable}
       Fix $\xi_1\in \BC$. Suppose $V_0 = V_{0,u}$, where $u$ runs in a fixed closed ball $\Omega = B(u_0,r_0)\in \BC^3$ ($u_0$ is the center) with a radius $r_0>0$.  Let $k_1$ be a sufficiently large positive integers such that
       \begin{equation*}
             -\mathrm{Re}(\xi_1)-k_1+2\mu<-1.
       \end{equation*}
       There is a finite set $C_u^{(1)}\sbst\bigcup_{j=1}^{k_1}E_{j,u}^{(2)}$, and a finite set $\mathcal{L}$ of nonnegative integers, a finite subset $\mathcal{P}_u\in \BC(u)[t]$ and a finite subset $\mathcal{D}\in U(\frak{gl}_3(\BC))$ such that for all $v\in V_{0,u}$, $x_1 \geq 0$, $W_{v,u}(e^{x_1H_1+x_2H_2})$ is a linear combination of terms of the following types:
       \begin{enumerate}
         \item  $e^{x_1\eta_u}p_u(x_1)W_{\pi(e)v,u}(e^{x_2H_2}),$
         \item  $e^{x_1\eta_u}h_u(x_1)\int_{0}^{x_1}e^{(-\eta_u-k_1)r}r^lW_{\pi(D)v,u}(e^{x_2H_2+rH_1})dr := h_u(x_1)\tilde{\phi}_{l,e,D}(x_1,x_2,v,u),$
         \item  $e^{x_1\eta_u}h_u(x_1)\int_0^{+\infty}e^{(-\eta_u-k_1)r}r^lW_{\pi(D)v,u}(e^{x_2H_2+rH_1})dr := e^{x_1\eta_u}h_u(x_1)\phi_{l,e,D}(x_1,x_2,v,u) ,$
         \item  $e^{x_1\eta_u}h_u(x_1)\int_{x_1}^{+\infty}e^{(-\eta_u-k_1)r}r^lW_{\pi(D)v,u}(e^{rH_1+x_2H_2})dr := h_u(x_1)\tilde{\phi}_{l,e,D}(x_1,x_2,v,u),$
       \end{enumerate}
       where $p_u, h_u\in\mathcal{P}_u$, $e,D \in \mathcal{D}$, and $\eta_u\in C_u^{(1)}$ are polynomials of $u$ of degree 1. For each $\eta_u\in C_u^{(1)}$, we decompose $\Omega = \Omega_{\eta_u,1}\bigcup\Omega_{\eta_u, 2}$ where
       \begin{equation*}
            \begin{aligned}
               \Omega_{\eta_u,1} := \set{u\in \Omega}{-k_1-\Re(\eta_u)+2\mu \geq -\frac{1}{2}};\\
               \Omega_{\eta_u,2} := \set{u\in \Omega}{-k_1-\Re(\eta_u)+2\mu \leq -\frac{1}{3}}.\\
            \end{aligned}
       \end{equation*}

       If $u\in \Omega_{\eta_u,1}$, then we only need terms of type (1) and type (2). The function $\phi_{l,e,D}(x_1,x_2,v,u)$ defined in type (2) is holomorphic in $u$ and uniformly continuous in $x_1,x_2,v$ when $u$ runs in $\Omega_{\eta_u,1}$. It satisfies the following estimate
       \begin{equation}\label{Eq: Estimate of First Asymp Exp 2nd term 2}
           \begin{aligned}
           &\big\lvert\tilde{\phi}_{l,e,D}(x_1,x_2,v,u)\big\rvert \leq e^{\Re(\xi_1)x_1}\cdot \tilde{\delta}_{l}(x_1)\cdot\norm{e^{x_2H_2}}^\mu \cdot q'(v)
           \end{aligned}
        \end{equation}
       where $q'$ is a continuous seminorm and $\tilde{\delta}_l(t) = \frac{t^{l+1}}{l+1}$.

       If $u\in \Omega_{\eta_u,2}$, we only need terms of type (1), (3) and (4). The functions\\ $\phi_{l,e,D}(x_1,x_2,v,u)$ and $\tilde{\phi}_{l,e,D}(x_1,x_2,v,u)$ are also holomorphic in $u\in \Omega_{\eta_u,2}$ and uniformly continuous in $x_1,x_2,v$ when $u$ runs in $\Omega_{\eta_u,2}$. They satisfy the following estimates:
       \begin{equation}\label{Eq: Estimate of First Asymp Exp 3rd term 2}
           \big\lvert \phi_{l,e,D}(x_1,x_2,v,u)\big\rvert \leq \delta_l(x_2)\cdot\norm{e^{x_1H_1}}^\mu\cdot q'(v);
       \end{equation}
       \begin{equation}\label{Eq: Estimate of First Asymp Exp 4th term 2}
           \big\lvert \tilde{\phi}_{l,e,D}(x_1,x_2,v,u)\big\lvert
           \leq e^{\Re(\xi_2)x_2}\tilde{\delta}_l(x_2)\cdot\norm{e^{x_1H_1}}^\mu\cdot q'(v);
       \end{equation}
       where $q'$ is a continuous seminorm, $\delta_l(t)$ is the constant function
       \begin{equation*}
            \delta_l(t) = \int_{0}^{+\infty}e^{-\frac{1}{3}t}t^ldt
       \end{equation*}
       and $\tilde{\delta}_l(t)$ is a polynomial of $t$ of degree $l$ with constant coefficients coming from integration by parts.

       Same expansions and estimates also hold if we drop the subscript $u$ and consider the underlying $(\frak{gl}_3,K_{GL_3})$-module for all irreducible generic Casselman-Wallach representations.
    \end{thm}

%%%%%%%%%%%%%%%%%%%%%%%%%%%%%%%%%%%%%%%%%%%%%%%%%%%%%%%%%%%%%%%%%%%%%%%%%%%%%%%%%%%%%%%%%%%%%%%%%%%%%%%%%%%%%%

    \subsection{Second Step}

%%%%%%%%%%%%%%%%%%%%%%%%%%%%%%%%%%%%%%%%%%%%%%%%%%%%%%%%%%%%%%%%%%%%%%%%%%%%%%%%%%%%%%%%%%%%%%%%%%%%%%%%%%%%%%%%%%
    From Theorem \ref{Theorem: Asymptotic Expansion GL_3}, we can see that certain summand in the asymptotic expansion will only appear for some certain $u$ (not all $u\in \Omega$). Hence to shorten the statement of our theorem, we say
    \begin{defn}
       A summand $S$ in the asymptotic expansion of the Whittaker function is said to be holomorphic in $u\in \Omega$ and uniformly continuous with respect to the other variables when $u$ runs in $\Omega$, if $S$ appears when $u$ lies inside a subset $\Omega_0\sbst \Omega$, and it is holomorphic in $u\in \Omega_0$ and uniformly continuous with respect to the other variables when $u$ runs in $\Omega_0$.
    \end{defn}
    The following is the main theorem of this Section.
    \begin{thm}\label{Theorem: Asymptotic Expansion GL_3}
       Fix $\xi_1$, $\xi_2$. Suppose $V_0 = V_{0,u}$ where $u$ runs in a fixed closed ball $\Omega = B(u_0,r_0)\in \BC^3$ ($u_0$ is the center) with a radius $r_0>0$. Let $k_1,k_2$ be two sufficiently large positive integers such that
       \begin{equation*}
          \begin{aligned}
             -\Re(\xi_1)-k_1+2\mu<-1,\\
             -\Re(\xi_2)-k_2+2\mu<-1.
          \end{aligned}
       \end{equation*}
       Then there are finite sets $C_u^{(1)}\sbst \bigcup_{j=1}^{k_1} E_j^{(1)}$ and $C_u^{(2)}\sbst\bigcup_{j=1}^{k_2}E_j^{(2)}$, and a finite set $\mathcal{P}_u\in \BC(u)[t_1,t_2]$, a finite subset $\mathcal{D}\in U(\frak{gl}_3(\BC))$ and a finite subset $\mathcal{L}$ of nonnegative integers such that for all $v\in V_{0,u}$, $x_1,x_2\geq 0$, $W_v(e^{x_1H_1+x_2H_2})$ is a linear combination of terms of the following types:
       \begin{enumerate}
         \item $e^{x_1\eta_u+x_2\xi_u}P_u(x_1,x_2)f_0(v,u),$
         \item $e^{x_1\eta_u}P_u(x_1,x_2)f_2(x_2,v,u),$
         \item $e^{x_2\xi_u}P_u(x_1,x_2)f_1(x_1,v,u),$
         \item $P_u(x_1,x_2)f_3(x_1,x_2,v,u),$
       \end{enumerate}
       where $\eta_u\in C_u^{(1)}$, $\xi_u\in C_u^{(2)}$ are polynomials of $u$ of degree 1, $P\in\mathcal{P}_u$, $f_0,f_1,f_2,f_3$ are holomorphic in $u$ and uniformly continuous in $x_1,x_2,v$ when $u$ runs in $\Omega$. They satisfy the following estimates

         \begin{enumerate}\label{Eq: Estimates of f in Asymptotic Expansions}
           \item  $\abs{f_0(v,u)}\leq q'(v)$,\\
           \item $\abs{f_1(x_1,v,u)}\leq e^{\text{Re}(\xi_1)\cdot x_1}h_1(x_1)q'(v)$,\\
           \item $\abs{f_2(x_2,v,u)}\leq e^{\text{Re}(\xi_2)\cdot x_2}h_2(x_2)q'(v)$,\\
            \item $\abs{f_3(x_1,x_2,v,u)}\leq e^{\text{Re}(\xi_1)\cdot x_1+\text{Re}(\xi_2)\cdot x_2}h_3(x_1,x_2)q'(v)$.\\
         \end{enumerate}

       Here in the above, $q'$ is a continuous seminorm on $V_\pi$ and $h_1,h_2,h_3$ are polynomials with complex coefficients. Same statements also hold when we drop the subscript $u$ and consider the underlying $(\frak{gl}_3,K_{\GL_3})$-module for all irreducible generic Casselman-Wallach representations.
    \end{thm}
    \begin{proof}
         For each $k_m (m=1,2)$, we choose $N(m)\in\BN$, $\{e_i^{(m)}\}_{i=1}^{N(m)}$, $\{D_{r,i}^{(m)}\}$, $B_u^{(m)}$ as in Lemma \ref{Lemma: Wallach 15.2.4},
        So far, we have proved that there exist finite sets $\mathcal{D}, \mathcal{P}_u, \mathcal{L},  C_u^{(2)}$ such that $W_{v,u}(e^{x_1H_1+x_2H_2})$ is a finite linear combination of
        \begin{equation*}
           e^{x_2\xi_u}p_{u}(x_2)W_{\pi(e)v,u}(e^{x_1H_1})
        \end{equation*}
        and
        \begin{equation*}
           e^{x_2\xi_u}h_u(x_2)\int_0^{x_2}e^{(-\xi_u-k_2)t}t^lW_{\pi(D)v,u}(e^{x_1H_1+tH_2})dt.
        \end{equation*}

        Because $x_1\geq 0$, by Theorem \ref{Theorem: Asymptotic Expansion first variable}, we can expand $W_{v,u}(e^{x_1H_1+tH_2})$ with respect to $x_1$. To save notations, we combine the finite set of non-negative integers, the finite set of the elements in the Lie algebra etc. Then by Theorem \ref{Theorem: Asymptotic Expansion first variable}, $W_{v,u}(e^{x_1H_1+tH_2})$ is a finite linear combination of terms of the following types:
        \begin{enumerate}
         \item \begin{equation}\label{Eq: Asymp x_1 1st term}
                  e^{x_1\eta_u}r_{u}(x_1)W_{\pi(e)v,u}(e^{tH_2})
               \end{equation}
         \item \begin{equation}\label{Eq: Asymp x_1 2nd term}
                  e^{x_1\eta_u}r_u(x_1)\int_0^{x_1}e^{(-\eta_u-k_1)r}r^{l'}W_{\pi(D)v,u}(e^{rH_1+tH_2})dr
               \end{equation}
         \item \begin{equation}\label{Eq: Asymp x_1 3rd term}
                  e^{x_1\eta_u}r_u(x_1)\int_{0}^{+\infty}e^{(-\eta_u-k_1)r}r^{l'}W_{\pi(D)v,u}(e^{rH_1+tH_2})dr
               \end{equation}
         \item \begin{equation}\label{Eq: Asymp x_1 4th term}
                  e^{x_1\eta_u}r_u(x_1)\int_{x_1}^{\infty}e^{(-\eta_u-k_1)r}r^{l'}W_{\pi(D)v,u}(e^{rH_1+tH_2})dr
               \end{equation}
         \end{enumerate}
         where $\eta_u\in C_u^{(1)}$, $r_u$ belongs to a finite subset $Q_u$ of $\BC(u)[t]$, $l'\in \mathcal{L}$, $D,e\in \mathcal{D}$.
         In particular, if we set $t=0$, we get an expansion of $W_{v,u}(e^{x_1H_1})$. Each term appears if $u$ satisfying the corresponding conditions in Theorem \ref{Theorem: Asymptotic Expansion first variable}. Thus
         \begin{equation*}
           e^{x_2\xi_u}p_{u}(x_2)W_{\pi(e)v,u}(e^{x_1H_1})
        \end{equation*}
         is a finite linear combination of
         \begin{equation*}
           e^{x_2\xi_u+x_1\eta_u}p_u(x_2)r_u(x_1)\phi_{l,e,D}(x_1,u,v)
         \end{equation*}
         and
         \begin{equation*}
           e^{x_2\xi_u}p_u(x_2)r_u(x_1)\tilde{\phi}_{l,e,D}(x_1,u,v),
         \end{equation*}
         where $\phi_{l,e,D}(x_1,u,v), \tilde{\phi}_{l,e,D}(u,v)$ are holomorphic in $u$ and uniformly continuous in $x_1,x_2,v$ when $u$ runs in $\Omega$. By the estimates in Theorem \ref{Theorem: Asymptotic Expansion first variable},
         \begin{equation*}
               \abs{\tilde{\phi}_{l,e,D}(x_1,u,v)}\leq e^{\Re(\xi_1)x_1}h_1(x_1)q'(v)
         \end{equation*}
         and
         \begin{equation*}
               \abs{\phi_{l,e,D}(u,v)}\leq q'(v)
         \end{equation*}
         for some polynomial $h_1$ with constant coefficients and continuous seminorm $q'$ ($h_1, q'$ may be different in different inequalities). Thus all summands in the expansion of \begin{equation*}
           e^{x_2\xi_u}p_{u}(x_2)W_{\pi(e)v,u}(e^{x_1H_1})
        \end{equation*}
        satisfy the properties in the theorem.

        As in the proof of Theorem \ref{Theorem: Asymptotic Expansion second variable}, for all $u\in \Omega$, $W_{v,u}(e^{x_1H_1+tH_2})$ is a finite linear combination of terms of the following types:
        \begin{enumerate}
         \item \begin{equation*}
                  e^{x_1\eta_u}r_{u}(x_1)W_{\pi(e)v,u}(e^{tH_2})
               \end{equation*}
         \item \begin{equation*}
                  e^{x_1\eta_u}r_u(x_1)\int_0^{x_1}e^{(-\eta_u-k_1)r}r^{l'}W_{\pi(D)v,u}(e^{rH_1+tH_2})dr.
               \end{equation*}
         \end{enumerate}
        Thus,
        \begin{equation*}
           e^{x_2\xi_u}h_u(x_2)\int_0^{x_2}e^{(-\xi_u-k_2)t}t^lW_{\pi(D)v,u}(e^{x_1H_1+tH_2})dt
        \end{equation*}
        is a finite linear combination of terms of follow types
        \begin{enumerate}
          \item
        \begin{equation}\label{Eq: 116}
               e^{x_1\eta_u+x_2\xi_u}r_{u}(x_1)h_u(x_2)\int_0^{x_2}e^{(-\xi_u-k_2)t}t^lW_{\pi(e)v,u}(e^{tH_2})dt;
        \end{equation}
        \item
        \begin{equation}\label{Eq: 117}
               e^{x_1\eta_u+x_2\xi_u}r_{u}(x_1)h_u(x_2)\int_0^{x_2}\int_0^{x_1} e^{(-\xi_u-k_2)t}t^l  e^{(-\eta_u-k_1)r}r^{l'} W_{\pi(D)v,u}(e^{rH_1+tH_2})dr dt.
        \end{equation}
        \end{enumerate}
        The summands in the expansion of (\ref{Eq: 116}) can be analyzed in the exact same way as in (\ref{Eq: First Asymp Exp 2nd term}). Thus those summands satisfy the properties in Theorem \ref{Theorem: Asymptotic Expansion GL_3}. It remains to deal with the integral (\ref{Eq: 117}). By using the same estimate method as in Theorem \ref{Theorem: Asymptotic Expansion second variable}, it is easy to see that the integral (\ref{Eq: 117}) defines a holomorphic function in $u\in \Omega$ which is uniformly continuous in $x_1,x_2,v$ when $u$ runs in $\Omega$. Let us define
        \begin{equation*}
            \begin{aligned}
               \Omega_{\xi_u,1} := \set{u\in \Omega}{-k_2-\Re(\xi_u)+2\mu \geq -\frac{1}{2}};\\
               \Omega_{\xi_u,2} := \set{u\in \Omega}{-k_2-\Re(\xi_u)+2\mu \leq -\frac{1}{3}};\\
               \Omega_{\eta_u,1}:= \set{u\in \Omega}{-k_1-\Re(\eta_u)+2\mu \geq -\frac{1}{2}};\\
               \Omega_{\eta_u,2} := \set{u\in \Omega}{-k_1-\Re(\eta_u)+2\mu \leq -\frac{1}{3}}.\\
            \end{aligned}
       \end{equation*}
       If $u\in\Omega_{\xi_u,1}$, we leave the $dt$-integral unchanged; if $u\in\Omega_{\xi_u,2}$, then we rewrite the $dt$-integral as $\int_{0}^{+\infty}-\int_{x_2}^{+\infty}$. We can do the similar operation for the $dr-$integral. Then we can imitate the method in Theorem \ref{Theorem: Asymptotic Expansion second variable} and finish the proof. Here we only deal one case, the others can be analyzed in the same way.

       Let us assume that $u\in \Omega_{\xi_u,1}\bigcap\Omega_{\eta_u,2}$, then (\ref{Eq: 117}) can be rewritten as
       \begin{equation}\label{Eq: 118}
          e^{x_1\eta_u+x_2\xi_u}r_{u}(x_1)h_u(x_2)\int_0^{x_2}\int_0^{+\infty} -   e^{x_1\eta_u+x_2\xi_u}r_{u}(x_1)h_u(x_2)\int_0^{x_2}\int_{x_1}^{+\infty}.
       \end{equation}
       Using a similar estimate as (\ref{Eq: 115}), we can show that the integral
       \begin{equation*}
          \int_0^{+\infty} e^{(-\eta_u-k_1)r}r^{l'} W_{\pi(D)v,u}(e^{rH_1+tH_2})dr
       \end{equation*}
      converges absolutely and is bounded by $\norm{e^{tH_2}}^\mu\cdot q'(v)$. Thus, both summands in (\ref{Eq: 118}) define a holomorphic function in $u$ which is uniformly continuous in the other variables when $u$ runs in $\Omega$. It suffices to prove that they satisfy the desired estimates. Indeed,
      \begin{equation*}
          \begin{aligned}
          &\big\lvert e^{x_2\xi_u}\int_0^{x_2}\int_0^{+\infty} e^{(-\xi_u-k_2)t}t^l  e^{(-\eta_u-k_1)r}r^{l'} W_{\pi(D)v,u}(e^{rH_1+tH_2})dr dt\big\lvert \\
          \leq &\big\lvert e^{x_2\xi_u}\cdot \int_0^{x_2} e^{(-\xi_u-k_2)t}t^l\cdot\norm{e^{tH_2}}^\mu q'(v)dt\big\lvert  .
          \end{aligned}
       \end{equation*}
      Thus, by the method in (\ref{Eq: 114}), the above is bounded by
      \begin{equation*}
          e^{\Re(\xi_2)x_2}h(x_2)q''(v)
      \end{equation*}
      for some polynomial $h$ with constant coefficients and a continuous seminorm $q''$. Therefore the first summand in (\ref{Eq: 118}) contributes a term of the form $e^{x_1\eta_u}p_u(x_1,x_2)f_2(x_2,v,u)$ in Theorem \ref{Theorem: Asymptotic Expansion GL_3}.

      By using a similar estimate as (\ref{Eq: Estimate of First Asymp Exp 4th term}), we can show that the integral
       \begin{equation*}
          \big\lvert e^{x_1\eta_u}\int_{x_1}^{+\infty} e^{(-\eta_u-k_1)r}r^{l'} W_{\pi(D)v,u}(e^{rH_1+tH_2})dr\big\lvert.
       \end{equation*}
      is bounded by $e^{\Re(\xi_1)x_1}h_1(x_1)q'(v)$ for some polynomial $h_1$ and continuous seminorm $q'$. Then using the method in (\ref{Eq: 114}), we can obtain that the integral
      \begin{equation*}
         \big\lvert e^{x_1\eta_u+x_2\xi_u}\int_0^{x_2}\int_{x_1}^{+\infty} e^{(-\xi_u-k_2)t}t^l  e^{(-\eta_u-k_1)r}r^{l'} W_{\pi(D)v,u}(e^{rH_1+tH_2})dr dt\big\lvert
      \end{equation*}
      is bounded by
      \begin{equation*}
          e^{\text{Re}(\xi_1)\cdot x_1+\text{Re}(\xi_2)\cdot x_2}h_3(x_1,x_2)q''(v)
      \end{equation*}
      for some continuous seminorm $q''$ and polynomial $h_3$ with complex coefficients. Thus the second summand in (\ref{Eq: 118}) contributes a term of the form $p_u(x_1,x_2)f_3(x_1, x_2,v,u)$ in Theorem \ref{Theorem: Asymptotic Expansion GL_3}.
    \end{proof}
\begin{rk}
   Since the space of $K$-finite vectors in $V_\pi$ is dense, by the continuity of Whittaker functional, the asymptotic expansion in Theorem \ref{Theorem: Asymptotic Expansion GL_3} also holds for all $v\in V_\pi$.
\end{rk}
%%%%%%%%%%%%%%%%%%%%%%%%%%%%%%%%%%%%%%%%%%%%%%%%%%%%%%%%%%%%%%%%%%%%%%%%%%%%%%%%

%%%%%%%%%%%%%%%%%%%%%%%%%%%%%%%%%%%%%%%%%%%%%%%%%%%%%%%%%%%%%%%%%%%%%%%%%%%%%%%%
\section{Meromorphic Continuation of the Local Integrals (Archimedean Case)} \label{section:  Mero}
%%%%%%%%%%%%%%%%%%%%%%%%%%%%%%%%%%%%%%%%%%%%%%%%%%%%%%%%%%%%%%%%%%%%%%%%%%%%%%%%
    In this Section, we will use Theorem \ref{Theorem: Asymptotic Expansion GL_3} to prove Theorem \ref{Theorem: Meromorphic Continuation} by analysing the asymptotic behaviour of the integrand of $Z(W_v, f_s)$ near 0. Throughout this Section, whenever an integral parameterized by $s$ is a bilinear form on the projective tensor product space $V_\pi\hat{\otimes} V_{\rho_s}$, we say that it satisfies property $\mathcal{M}$, if
    \begin{enumerate}
       \item it converges absolutely when Re$(s)$ is sufficiently large,
       \item it extends to a meromorphic function of $s$ to the whole complex plane,
       \item its meromorphic continuation is a continuous on $V_{\pi}\hat{\otimes} V_{\rho_s}$.
    \end{enumerate}
    For some basic properties of projective tensor space, we refer to \cite{Tr}. Now we start from some reductions to get rid of the integral over the maximal compact subgroup.

%%%%%%%%%%%%%%%%%%%%%%%%%%%%%%%%%%%%%%%%%%%%%%%%%%%%%%%%%%%%%%%%%%%%%%%%%%%%%%%%%%%%%%%%%%%%%%%%%%%%%%%%%%%%%%%%

    \subsection{Some Reductions}

%%%%%%%%%%%%%%%%%%%%%%%%%%%%%%%%%%%%%%%%%%%%%%%%%%%%%%%%%%%%%%%%%%%%%%%%%%%%%%%%%%%%%%%%%%%%%%%%%%%%%%%%%%%%%%%%%
    Let us first define
    \begin{equation*}
       \begin{aligned}
       B(W_v, f_s) &:= \int_F\int_{A_{\SL_3}} W_v(\mtrthr{1}{z}{}{}{1}{}{}{}{1}a)f_s(\gamma\cdot \mtrthr{1}{z}{}{}{1}{}{}{}{1}a)\delta_{B_{\SL_3}}^{-1}(a)dadz,
       \end{aligned}
    \end{equation*}
    where the subgroup $A_{\SL_3}$ is defined in \eqref{eq: positive torus}. Then the local integral $Z(W_v, f_s)$ defined in \eqref{Def: local archi integral} can be rewritten as
    \begin{equation*}
       Z(W_v, f_s) = \int_{K_{\SL_3}} B(\pi(k)W_v, \rho_s(k)f_s)dk.
    \end{equation*}
    \begin{lemma}
       If $B(W_v,f_s)$ satisfies property $\mathcal{M}$, then $Z(W_v,f_s)$ satisfies property $\mathcal{M}$. Moreover, if $\pi = \pi_u$ is a principal series as in \eqref{eq: principal series} and we assume that  $B(W_{v,u},f_s)$ is meromorphic in $u$, then $Z(W_{v,u},f_s)$ is also meromorphic in $u$.
    \end{lemma}
    \begin{proof}
       If $B(W_v,f_s)$ satisfies property $\mathcal{M}$, then whenever $s=s_0$ belongs to a compact set away from the poles of $Z(W_v,f_s)$, the function
       \begin{equation*}
          k\quad\mapsto\quad B(\pi(k)W_v, \rho_s(k)f_s),
       \end{equation*}
       with both $v\in V_\pi$ and $f_s\in V_{\rho_s}$ fixed, is a continuous function on a compact group, hence it is bounded. Moreover, the integral
       \begin{equation}\label{Eq: 119}
           Z(W_v, f_s) = \int_{K_{\SL_3}} B(\pi(k)W_v, \rho_s(k)f_s)dk
       \end{equation}
       converges absolutely and uniformly in $s$ when $s$ runs in that compact set. This proves the meromorphic continuation of $Z(W_v,f_s)$. Similarly, when $\pi = \pi_u$, the convergence of (\ref{Eq: 119}) is also uniform in $u$ when $u$ runs in a compact set $\Omega$ away from the poles. Thus, if $B(W_{v,u},f_s)$ is meromorphic in $u$, then $Z(W_{v,u},f_s)$ is also meromorphic in $u$. For each fixed $k\in K_{\SL_3}$, the function
       \begin{equation*}
          (v,f_s)\quad\mapsto\quad B(\pi(k)W_v, \rho_s(k)f_s)
       \end{equation*}
       is a bounded continuous bilinear form on $V_\pi\hat{\otimes} V_{\rho_s}$. Set
       \begin{equation*}
          B_k(v,f_s) = B(\pi(k)W_v, \rho_s(k)f_s).
       \end{equation*}
       Then by the Uniform Boundedness Principle (see \cite{Tr}), the family of bounded continuous bilinear form $B_k$ indexed by $k\in K_{\SL_3}$ is equicontinuous. In other words, let $d_{\pi,\rho}$ be the metric describing the topology of the Fr\'{e}chet space $V_\pi\hat{\otimes} V_{\rho_s}$ , then for any $\eps >0$, there exists $\delta>0$ such that for any two pairs $(v_1, f_{1,s}), (v_2, f_{2,s})\in V_\pi\hat{\otimes} V_{\rho_s}$ satisfying
       \begin{equation*}
          d_{\pi,\rho}((v_1, f_{1,s}), (v_2, f_{2,s}))<\delta,
       \end{equation*}
       we have
       \begin{equation*}
          \abs{B(\pi(k)W_{v_1}, \rho_s(k)f_{1,s})-B(\pi(k)W_{v_2}, \rho_s(k)f_{2,s})}<\eps.
       \end{equation*}
       Therefore
       \begin{equation*}
          \begin{aligned}
          \abs{Z(W_{v_1},f_{1,s})-Z(W_{v_2},f_{2,s})} &\leq \int_{K_{\SL_3}} \abs{B(\pi(k)W_{v_1}, \rho_s(k)f_{1,s})-B(\pi(k)W_{v_2}, \rho_s(k)f_{2,s})}dk\\
          &\leq \eps.
          \end{aligned}
       \end{equation*}
       This shows that $Z(W_v, f_s)$ is a continuous bilinear form.
    \end{proof}
    The following Dixmier-Malliavin lemma (see \cite{D-M}) is well known to experts.
    \begin{lemma}
       Let $(\pi, V)$ be a continuous representation of a Lie group $\RG$ on a Fr\'{e}chet space $V$. Then every smooth vector $v\in V$ can be represented by a finite linear combination
       \begin{equation*}
            v = \sum_{i} \pi(f_i)v_i = \sum_i \int_{\RG} \pi(x)f_i(x)v_i dx,
       \end{equation*}
       where $f_i(x)\in C^\infty_c(\RG)$, and $v_i$ are smooth vectors in $V$.
    \end{lemma}

    We apply the Dixmier-Malliavin lemma to the subgroup $U_1=\Big\{\mtrthr{1}{x}{}{}{1}{}{}{}{1}\Big\}$. Every $v\in V_\pi$ is a finite linear combination of $\pi(\varphi_i^{(1)})v_i$:
    \begin{equation*}
       v=\sum_i \pi(\varphi_i^{(1)})v_i,
    \end{equation*}
    where $\varphi_i^{(1)}\in C^\infty_c(F)$ and each $v_i \in V_\pi$. Thus, accordingly,
    \begin{equation}\label{eq: update 001}
       W_v(\mtrthr{t_1^3t_2^3}{}{}{}{t_2^3}{}{}{}{1}) = \sum_i W_{\pi(\varphi_i^{(1)})v_i}(\mtrthr{t_1^3t_2^3}{}{}{}{t_2^3}{}{}{}{1}) = \sum_i W_{v_i}(\mtrthr{t_1^3t_2^3}{}{}{}{t_2^3}{}{}{}{1})\widehat{\varphi_i^{(1)}}(t_1^3),
    \end{equation}
    where $\widehat{\varphi_i^{(1)}}$ is a Fourier transform of $\varphi_i^{(1)}$, hence a Schwartz function.
    We apply the Dixmier-Malliavin lemma to the subgroup $U_2=\Big\{\mtrthr{1}{}{}{}{1}{x}{}{}{1}\Big\}$ and apply the same trick as in \eqref{eq: update 001}. Then we can conclude that the function
    $$W_v(\mtrthr{t_1^3t_2^3}{}{}{}{t_2^3}{}{}{}{1})$$
    is a finite linear combination of
    \begin{equation}\label{eq: update 02}
       W_{v_j}(\mtrthr{t_1^3t_2^3}{}{}{}{t_2^3}{}{}{}{1}){\varphi^{(1)}_j}(t_1^3)\varphi^{(2)}_j(t_2^3),
    \end{equation}
    where $v_j\in V_\pi$, $\varphi_j^{(1)}$ and $\varphi_j^{(2)}$ are Schwartz functions on $F$. Using the same calculation as in the proof of Theorem \ref{Theorem: Absolute Convergence of Integral} (see \eqref{Eq:07}), we can obtain that
    \begin{equation}\label{Eq: Expression of B}
      \begin{aligned}
       B(W_v,f_s)
                  &= \int_F\int_{(\BR_+^\times)^2} W_v(\mtrthr{t_1^3t_2^3}{}{}{}{t_2^3}{}{}{}{1})\psi((t_1^2+t_2^2)^{\frac{3}{2}}z)\cdot\Big\lvert\frac{t_1^{9s}t_2^{9s}}{(t_1^2+t_2^2)^{\frac{9}{2}s}}\Big\lvert_F\cdot\abs{t_1^{-6}t_2^{-6}}_F \\
           &\qquad\qquad \cdot\abs{\abs{z}^2+1}_F
^{-\frac{3}{2}s}\cdot\abs{t_1^2+t_2^2}_F^{\frac{3}{2}}\cdot 3f_s(k''(z)k'(t_1^{-1}t_2)w_\beta )d^\times t_1d^\times t_2dz.
       \end{aligned}
    \end{equation}

    Because both $V_\pi$ and the space of Schwartz functions on $F$, denoted by $\mathcal{S}(F)$, are Fr\'{e}chet spaces and the bilinear map $\mathcal{S}(F)\times V_\pi \rightarrow V_\pi$:
    \begin{equation*}
        (\varphi, v)\mapsto \pi(\varphi)v
    \end{equation*}
    is separately continuous, we can conclude that the linear map $\mathcal{S}(F)\hat{\otimes}V_\pi\mapsto V_\pi$:
    \begin{equation*}
        \varphi\hat{\otimes}v \mapsto \pi(\varphi)v
    \end{equation*}
    is continuous and surjective (by Deximier-Malliavin Lemma). Hence it is also an open map by the Open Mapping Theorem.

    Now we combine \eqref{eq: update 02}, \eqref{Eq: Expression of B} and conclude that: to prove that $B(W_{v,u},f_s)$ satisfies the property $\mathcal{M}$ and  is meromorphic in $u$, it suffices to prove that the following integral $C(W_{v,u},f_s)$ satisfies property $\mathcal{M}$ and is meromorphic in $u$:
    \begin{equation}\label{Eq: Expression of C}
       \begin{aligned}
       C(W_{v,u},f_s) &:= \int_{F}\int_0^{+\infty}\int_0^{+\infty} W_{v,u}(\mtrthr{t_1^3t_2^3}{}{}{}{t_2^3}{}{}{}{1})\psi(z)\cdot \abs{t_1^{9s-6}t_2^{9s-6}}_F\cdot\varphi_1(t_1)\varphi_2(t_2) \\
           &\qquad\qquad\cdot\abs{\abs{z}^2+(t_1^2+t_2^2)^{3}}_F
^{-\frac{3}{2}s}\cdot 3f_s(k''(\frac{z}{(t_1^2+t_2^2)^{\frac{3}{2}}})k'(t_1^{-1}t_2)w_\beta )d^\times t_1d^\times t_2dz,\\
       \end{aligned}
    \end{equation}
    where $\varphi_1$,$\varphi_2\in \CS(F)$.
    \begin{rk}
       The goal to introduce Schwartz functions $\varphi_1$, $\varphi_2$ is to control the behaviour of the integrand of the RHS of \eqref{Eq: Expression of C} when $t_1$ or $t_2$ goes to infinity.
    \end{rk}

%%%%%%%%%%%%%%%%%%%%%%%%%%%%%%%%%%%%%%%%%%%%%%%%%%%%%%%%%%%%%%%%%%%%%%%%%%%%%%%%%%%%%%%%%%%%%%%%%%%%%%%%%%%%%%%%%%%%%%%%%%%%%%%%%%%%%%%%%%%%%%%%

    \subsection{Proof of Theorem \ref{Theorem: Meromorphic Continuation}}

%%%%%%%%%%%%%%%%%%%%%%%%%%%%%%%%%%%%%%%%%%%%%%%%%%%%%%%%%%%%%%%%%%%%%%%%%%%%%%%%%%%%%%%%%%%%%%%%%%%%%%%%%%%%%%%%%%%%%%%%%%%%%%%%%%%%%%%%%%%%%%%%
    We only focus on the real case ($F=\BR$), the proof of the complex case is similar. Note that by the Iwasawa decomposition of $\SL_2(\BR)$,
    \begin{equation*}
       \mtrtwo{1}{}{t_1^{-1}t_2}{1} = \mtrtwo{\frac{t_1}{\sqrt{t_1^2+t_2^2}}}{}{}{\frac{\sqrt{t_1^2+t_2^2}}{t_1}}\mtrtwo{1}{t_1^{-1}t_2}{}{1}\mtrtwo{\frac{t_1}{\sqrt{t_1^2+t_2^2}}}{\frac{-t_2}{\sqrt{t_1^2+t_2^2}}}{\frac{t_2}{\sqrt{t_1^2+t_2^2}}}{\frac{t_1}{\sqrt{t_1^2+t_2^2}}}.
    \end{equation*}
    Hence as a function of two variables $t_1, t_2$, $k'$ is a smooth bounded function on $[0,+\infty)\times [0,+\infty)-\{(0,0)\}$. If we set $t_1 = r\cos\theta, t_2=r\sin\theta$, then variable $r$ doesn't appear in $k'$. As a function of $\theta$, $k'$ can be extended to a smooth function on $[0,2\pi]$.

    Let us first deal with the $dz$-integral first.

    Set
    \begin{equation}\label{Eq: Def of F(t1,t2,s)}
       F(t_1,t_2,s) := 3(t_1^2+t_2^2)^{\frac{9}{4}s}\int_\BR \psi(z)\cdot (\abs{z}^2+(t_1^2+t_2^2)^3)^{-\frac{3}{2}s}f_s(k''(\frac{z}{(t_1^2+t_2^2)^{\frac{3}{2}}})k'(t_1^{-1}t_2)w_\beta)dz.
    \end{equation}
    \begin{lemma}\label{Lemma: Holomorphic continuation of dz integral}
       For any $(t_1,t_2)\ne (0,0)$, $F(t_1,t_2,s)$ converges absolutely when $\text{Re}(s)>\frac{1}{3}$, and it has a holomorphic continuation in the whole complex plane.
       Moreover, if we use polar coordinate
       \begin{equation*}
          t_1 = r\cos\theta, t_2 = r\sin\theta,
       \end{equation*}
       then as a function of $r$ and $\theta$, $F(t_1(r,\theta),t_2(r,\theta),s)$ (simply denoted by $F(r,\theta,s)$ in the following) is a bounded smooth function in $\theta$ and behaves like a Schwartz function when $r$ tends to infinity. $F(r,\theta,s)$ also has an asymptotic expansion
       \begin{equation*}
        F(r,\theta,s) \sim \sum_{k=0}^{+\infty} a_k(\theta)r^{\frac{9}{2}s+6k}+\sum_{k=0}^{+\infty} b_k(\theta)r^{3-\frac{9}{2}s+6k}.
    \end{equation*}
       when $r$ tends to zero, where $a_k(\theta),b_k(\theta)$ are bounded smooth functions of $\theta$.
    \end{lemma}
    \begin{proof}
    For every fixed $(t_1,t_2)\ne (0,0)$, $F(t_1,t_2,s)$ converges absolutely when $\text{Re}(s)>\frac{1}{3}$, hence it defines a holomorphic function in the right half plane $\text{Re}(s)>\frac{1}{3}$.

    If we set $t_1 = r\cos\theta, t_2 = r\sin\theta$, then
    \begin{equation*}
        F(r,\theta,s) = 3r^{\frac{9}{2}s}\int_\BR \psi(z)\cdot (z^2+r^6)^{-\frac{3}{2}s}f_s(k''(\frac{z}{r^3})k'(\theta)w_\beta)dz.
    \end{equation*}
    $\SL_3(\BR)$ contains a subgroup $H$ that is isomorphic to $\SL_2(\BR)$ lying on the upper left corner:
    \begin{equation*}
        H := \set{\mtrtwo{g}{}{}{1}}{g\in \SL_2(\BR)}.
    \end{equation*}
    The restriction of any function $f_s\in V_{\rho_s}$ on $H$ satisfies the following invariant property:
    \begin{equation*}
       f_s(\mtrthr{1}{z}{}{}{1}{}{}{}{1}\mtrthr{t}{}{}{}{t^{-1}}{}{}{}{1}x) = \abs{t}^{3s}f_s(x),
    \end{equation*}
    hence $f_s\lvert_H$ lies inside a principal series of $\SL_2(\BR)$:
    $$\Ind_{B_{\SL_2}}^{\SL_2(\BR)} \abs{\quad}^{\frac{3s-1}{2}}\otimes \abs{\quad}^{-\frac{3s-1}{2}}.$$
    When we fix $\theta$, then $F(r,\theta,s)$ is exactly the Jacquet integral of a right translation of $f_s$. Hence by \cite[Section 15.4]{Wal2}, the above integral has a holomorphic continuation to the whole complex plane, and it behaves like a Schwartz function when $r$ tends to infinity. Moreover, $F(r,\theta,s)$ has the following asymptotic expansion when $r$ tends to zero:
    \begin{equation*}
        F(r,\theta,s) \sim \sum_{k=0}^{+\infty} a_k(\theta)r^{\frac{9}{2}s+6k}+\sum_{k=0}^{+\infty} b_k(\theta)r^{3-\frac{9}{2}s+6k}.
    \end{equation*}
    Because all the asymptotic coefficients $a_k(\theta), b_k(\theta)$ can be obtained by limit process recursively (see \cite[Section 1.4]{Bl-H}), they are bounded smooth functions of $\theta$.
    \end{proof}

    Now we begin to prove the desired meromorphic continuation of $C(W_{v,u},f_s)$. In the real case, the integral $C(W_{v,u},f_s)$ defined in \eqref{Eq: Expression of C} can be rewritten as
    \begin{equation*}
        \begin{aligned}
           C(W_{v,u},f_s) &= \int_0^{+\infty}\int_0^{+\infty} W_{v,u}(\mtrthr{t_1^3t_2^3}{}{}{}{t_2^3}{}{}{}{1})\cdot\frac{t_1^{9s-6}t_2^{9s-6}}{(t_1^2+t_2^2)^{\frac{9}{4}s}}\\
           &\quad\cdot\varphi_1(t_1)\varphi_2(t_2)\cdot F(t_1,t_2,s)
            d^\times t_1d^\times t_2.
        \end{aligned}
    \end{equation*}
    We note that $F(t_1,t_2,s)$ is smooth on $[0,+\infty)\times [0,\infty)-\{(0,0)\}$. At the origin $(0,0)$, we can only expect an asymptotic expansion as above.

    \begin{proof}[Proof of Theorem \ref{Theorem: Meromorphic Continuation}, real case]
        Still, because we want to keep track on the parameters $u$, we only focus on the case $\pi = \pi_u$. A word by word repetition will also work for any irreducible generic Casselman-Wallach representation $\pi$. Let $u$ run in a fixed closed ball $\Omega = B(u_0,r_0)$. Since both $V_{\pi_u}$ and $V_{\rho_s}$ are Fr\'{e}chet spaces, the notions of continuity and separate continuity on $V_{\pi_u}\hat{\otimes}V_{\rho_s}$ coincide. So we can fix $f_s$ first and prove the continuity in $v$. We fix a vertical strip Re$(s)\in (a,b)$ (here $a<b$ are two real numbers). For $m=1,2$, we choose two negative numbers $\xi_1,\xi_2$ which has sufficiently large absolute values. The exact conditions that they satisfy will be clear from the following proof.

        We choose $k_1,k_2$ so large such that
        \begin{equation*}
          \begin{aligned}
            &-\text{Re}(\xi_1)-k_1+2\mu<-1,\\
            &-\text{Re}(\xi_2)-k_2+2\mu<-1.
          \end{aligned}
       \end{equation*}
       Then there exist two finite subsets $C_u^{(1)}, C_u^{(2)}$, a finite set of polynomials $\mathcal{P}_u$, a finite subset of non-negative integers $\mathcal{L}$ and a finite subset $\mathcal{D}\sbst U(\mathfrak{gl}_3(\BC))$ such that $W_{v,u}(e^{x_1H_1+x_2H_2})$ has an asymptotic expansion as in Theorem $\ref{Theorem: Asymptotic Expansion second variable}$ when $x_2 \geq 0$, and $W_v(e^{x_1H_1+x_2H_2})$ has an asymptotic expansion as in Theorem \ref{Theorem: Asymptotic Expansion GL_3} when $x_1,x_2 \geq 0$.

       We fix a very small positive number $1>\eps>0$ and split the integral $C(W_{v,u},f_s)$ into four parts:
        \begin{equation}\label{eq: 12345}
           \int_\eps^{+\infty}\int_\eps^{+\infty},\qquad\int_0^{\eps}\int_\eps^{+\infty},\qquad\int_\eps^{+\infty}\int_0^{\eps},\qquad\int_0^{\eps}\int_0^{\eps}.
        \end{equation}

        Case 1: $t_1,t_2\geq \eps$.
        Since both $\varphi_1(t_1)$ and $\varphi_2(t_2)$ are Schwartz functions and $F(t_1,t_2,s)$ behaves like a Schwartz function when $t_1^2+t_2^2 \rightarrow +\infty$, the integral
        \begin{equation*}
           \int_\eps^{+\infty}\int_\eps^{+\infty} W_{v,u}(\mtrthr{t_1^3t_2^3}{}{}{}{t_2^3}{}{}{}{1})\cdot\frac{t_1^{9s-6}t_2^{9s-6}}{(t_1^2+t_2^2)^{\frac{9}{4}s}}\cdot
           \cdot\varphi_1(t_1)\varphi_2(t_2)\cdot F(t_1,t_2,s) d^\times t_1d^\times t_2
        \end{equation*}
        converges absolutely and defines a holomorphic function of $s$. Moreover, by the estimate (\ref{Eq: 08}), the above integral defines a holomorphic function in $u\in\Omega$ which is continuous in $v$.

        Case 2: $t_1\geq \eps, t_2\leq \eps$.
        We apply the asymptotic expansion of $W_{v,u}$ in the variable $t_2$. Then by Theorem \ref{Theorem: Asymptotic Expansion second variable},  $W_{v,u}(\mtrthr{t_1^3t_2^3}{}{}{}{t_2^3}{}{}{}{1})$ is a finite $\BC(u)$-linear combination of terms of the form
        \begin{enumerate}
          \item type 1: $t_2^{-3\xi_u}(\ln t_2)^r\cdot W_{\pi(e)v,u}(\mtrthr{t_1^3}{}{}{}{1}{}{}{}{1})$,
          \item type 2: $t_2^{-3\xi_u}(\ln t_2)^r\cdot \int_0^{+\infty} e^{(-\xi-k_2)t}t^l W_{\pi(D)v,u}(\mtrthr{t_1^3}{}{}{}{1}{}{}{}{1}\cdot e^{tH_2})dt$,
          \item type 3: $t_2^{-3\xi_u}(\ln t_2)^r\cdot \int_{-3\ln t_2}^{+\infty} e^{(-\xi-k_2)t}t^l W_{\pi(D)v,u}(\mtrthr{t_1^3}{}{}{}{1}{}{}{}{1}\cdot e^{tH_2})dt$,
          \item type 4: $t_2^{-3\xi_u}(\ln t_2)^r\cdot \int_{0}^{-3\ln t_2} e^{(-\xi-k_2)t}t^l W_{\pi(D)v,u}(\mtrthr{t_1^3}{}{}{}{1}{}{}{}{1}\cdot e^{tH_2})dt$,
        \end{enumerate}
        where all $e, D\in\mathcal{D}$, $r, l\in\mathcal{L}$, $\xi_u\in C_u^{(2)}$. To unify our notations, each term above can be written as
        \begin{equation*}
           t_2^{-3\xi_u}(\ln t_2)^r\cdot H(t_1,t_2,v,u),
        \end{equation*}
        where $H(t_1,t_2,v,u)$ is holomorphic in $u\in\Omega$ and uniformly continuous in $t_1,t_2,v$ when $u$ runs in $\Omega$. It also satisfies the estimates in Theorem \ref{Theorem: Asymptotic Expansion second variable}. Put
        \begin{equation*}
           G(t_2,v,u,s) := \int_\eps^{+\infty} H(t_1,t_2,v,u)\frac{t_1^{9s-6}}{(t_1^2+t_2^2)^{\frac{9}{4}s}}\varphi_1(t_1)F(t_1,t_2,s)d^\times t_1.
        \end{equation*}
       Then for each fixed $t_2$, the integral $G(t_2,v,u,s)$ converges absolutely for all $s$ and $u\in \Omega$, hence it defines a holomorphic function in $s$ and $u$. It is also a continuous in $t_2$ and linear in $v$. The second part of $C(W_{v,u},f_s)$ in \eqref{eq: 12345} is a finite linear combination of integrals of the form
       \begin{equation}\label{Eq: 12}
          \int_{0}^\eps t_2^{9s-6-3\xi_u}(\ln t_2)^r\varphi_2(t_2)G(t_2,v,u,s) d^\times t_2.
       \end{equation}

       If $H(t_1,t_2, u, v)$ is of type 1 or 2, then it is in fact a continuous function in $t_1$ ($t_2$ does not appear in $H$) bounded by $\norm{e^{x_1H_1}}^\mu q'(v)$ for some continuous seminorm $q'$ (by Theorem \ref{Theorem: Asymptotic Expansion second variable}). In these two cases, from the expression
       \begin{equation*}
          G(t_2,v,u,s) = \int_\eps^{+\infty} H(t_1,v,u)\frac{t_1^{9s-6}}{(t_1^2+t_2^2)^{\frac{9}{4}s}}\varphi_1(t_1)F(t_1,t_2,s)d^\times t_1,
       \end{equation*}
       $G(t_2,v,u,s)$ and all its partial derivatives $\frac{\partial^l G}{\partial t_2^l}(t_2,v,u,s)$ are smooth functions in $t_2$, holomorphic in $u$, and continuous and linear in $v$. Therefore, the integral
       \begin{equation*}
          \int_{0}^\eps t_2^{9s-6-3\xi_u}(\ln t_2)^r\varphi_2(t_2)G(t_2,v,u,s) d^\times t_2
       \end{equation*}
       admits a meromorphic continuation in $s$. This meromorphic continuation is also meromorphic in $u$ and continuous in $v$ (since each $\xi_u$ is a polynomial of $u$ of degree 1).

       If $H(t_1,t_2, u, v)$ is of type 3, then by the estimates in Theorem \ref{Theorem: Asymptotic Expansion second variable},
       \begin{equation*}
          \begin{aligned}
          &\abs{t_2^{-3\xi_u}(\ln t_2)^rG(t_2,v,u,s)}\\
          \leq  &\int_\eps^{+\infty} \big\lvert t_2^{-3\xi_u}(\ln t_2)^r H(t_1,t_2,u,v)\frac{t_1^{9s-6}}{(t_1^2+t_2^2)^{\frac{9}{4}s}}\varphi_1(t_1)F(t_1,t_2,s)\big\rvert d^\times t_1\\
          \leq &\int_\eps^{+\infty}\big\lvert t_2^{-3\text{Re}(\xi_2)}\cdot t_1^{6\mu}h(\ln t_2)q'(v)\cdot\frac{t_1^{9s-6}}{(t_1^2+t_2^2)^{\frac{9}{4}s}}\varphi_1(t_1)F(t_1,t_2,s)\big\rvert d^\times t_1\\
          \leq &t_2^{-3\text{Re}(\xi_2)}h(\ln t_2)q''(v)
          \end{aligned}
       \end{equation*}
       for some polynomial $h$ and continuous seminorm $q''$. Therefore
       \begin{equation*}
          \abs{t_2^{9s-6-3\xi_u}(\ln t_2)^r\varphi_2(t_2)G(t_2,v,s)}\leq t_2^{9\text{Re}(s)-6-3\text{Re}(\xi_2)}h(\ln t_2)\abs{\varphi_2(t_2)}q''(v)
       \end{equation*}
       for some polynomial $h$ and continuous seminorm $q''$. Since we can choose Re$(\xi_2)$ as negative as we want, if we assume that the exponent satisfies $$9\text{Re}(s)-6-3\text{Re}(\xi_2)>1$$ in the vertical strip Re$(s)\in (a,b)$, then (\ref{Eq: 12}) is holomorphic in $s$ and $u$, and defines a continuous linear function in $v$.

       If $H(t_1,t_2,u,v)$ is of type 4, the estimate for type 3 also holds in this case, since the estimate for 2) and 4) in Theorem \ref{Theorem: Asymptotic Expansion second variable} share the same pattern. Thus, the integral
       \begin{equation*}
          \int_{0}^\eps t_2^{9s-6-3\xi_u}(\ln t_2)^r\varphi_2(t_2)G(t_2,v, u, s) d^\times t_2
       \end{equation*}
       also defines a holomorphic function in $s$ and $u$, which is also continuous in $v$.

        Case 3: $t_1\leq \eps, t_2\geq \eps$.
        Since $t_1$ and $t_2$ play a symmetric role in the integral, the proof is exactly the same as Case 2.

        Case 4: $t_1<\eps, t_2<\eps$.
        We apply Theorem $\ref{Theorem: Asymptotic Expansion GL_3}$ again. When $t_1<\eps, t_2<\eps$ , $W_v(\mtrthr{t_1^3t_2^3}{}{}{}{t_2^3}{}{}{}{1})$ is a finite $\BC(u)$-linear combination of terms of the form
        \begin{enumerate}
          \item $t_1^{-3\eta_u}t_2^{-3\xi_u}(\ln t_1)^{r_1}(\ln t_2)^{r_2}f_0(v,u),$
          \item $t_1^{-3\eta_u}(\ln t_1)^{r_1}(\ln t_2)^{r_2}f_2(\ln t_2, v,u),$
          \item $t_2^{-3\xi_u}(\ln t_1)^{r_1}(\ln t_2)^{r_2}f_1(\ln t_1, v,u),$
          \item $(\ln t_1)^{r_1}(\ln t_2)^{r_2}f_3(\ln t_1, \ln t_2,v,u),$
        \end{enumerate}
        where $r_1, r_2 \in \mathcal{L}$, $\eta_u\in C_u^{(1)}$, $\xi_u\in C_u^{(2)}$. Moreover, $f_0,f_1,f_2$ admit estimates as in Theorem $\ref{Theorem: Asymptotic Expansion GL_3}$, i.e
        \begin{enumerate}
            \item $\abs{f_0(v,u)}\leq q'(v),$
            \item $\abs{f_1(\ln t_1,v,u)}\leq t_1^{-3\text{Re}(\xi_1)}h_1(\ln t_1)q'(v),$
            \item $\abs{f_2(\ln t_2,v,u)}\leq t_2^{-3\text{Re}(\xi_2)}h_2(\ln t_2)q'(v),$
            \item $\abs{f_3(\ln t_1,\ln t_2,v,u)}\leq t_1^{-3\text{Re}(\xi_1)}t_2^{-3\text{Re}(\xi_2)}h_3(\ln t_1,\ln t_2)q'(v).$
        \end{enumerate}
        We need to study
        \begin{equation}\label{Eq: 13}
           \int_0^\eps\int_0^\eps W_{v,u}(\mtrthr{t_1^3t_2^3}{}{}{}{t_2^3}{}{}{}{1})\cdot\frac{t_1^{9s-6}t_2^{9s-6}}{(t_1^2+t_2^2)^{\frac{9}{4}s}}\cdot\varphi_1(t_1)\varphi_2(t_2)F(t_1,t_2,s) d^\times t_1d^\times t_2.
        \end{equation}
        We can choose $\eps$ so small that $F(t_1,t_2,s)$ (or equivalently $F(r,\theta,s)$) can be approximated by its asymptotic expansion (see Lemma \ref{Lemma: Holomorphic continuation of dz integral}).
        \begin{equation*}
        F(r,\theta,s) \sim \sum_{k=0}^{+\infty} a_k(\theta)r^{\frac{9}{2}s+6k}+\sum_{k=0}^{+\infty} b_k(\theta)r^{3-\frac{9}{2}s+6k}.
    \end{equation*}

    After we cut off the first finite terms $F_q(r,\theta,s)$ in the asymptotic expansion of $F(r,\theta,s)$, the remainder
    \begin{equation*}
         F(r,\theta,s)-F_q(r,\theta,s) = O(r^p).
    \end{equation*}
    Here $p$ can be as large as we want if we choose a sufficiently large $q$. Now we use the estimate
    \begin{equation*}
        \abs{W_{v,u}(\mtrthr{t_1^3t_2^3}{}{}{}{t_2^3}{}{}{}{1})} \leq 6^{\mu}t_1^{-6\mu}t_2^{-6\mu}q(v).
    \end{equation*}
    If we choose a large $p$ such that $p$ can beat any exponent appearing in the integrand of (\ref{Eq: 13}), then
    \begin{equation*}
        \int_0^\eps\int_0^\eps W_{v,u}(\mtrthr{t_1^3t_2^3}{}{}{}{t_2^3}{}{}{}{1})\cdot\frac{t_1^{9s-6}t_2^{9s-6}}{(t_1^2+t_2^2)^{\frac{9}{4}s}}\cdot\varphi_1(t_1)\varphi_2(t_2)(F(t_1,t_2,s)-F_q(t_1,t_2,s)) d^\times t_1d^\times t_2
    \end{equation*}
    is a holomorphic function in $s$ and $u$, and continuous in $v$. We note that the choice of $p$ (hence $q$) only depends on the vertical strip $(a,b)$. Therefore it suffices to prove the meromorphic continuation of
    \begin{equation*}
        \int_0^\eps\int_0^\eps W_{v,u}(\mtrthr{t_1^3t_2^3}{}{}{}{t_2^3}{}{}{}{1})\cdot\frac{t_1^{9s-6+l_1}t_2^{9s-6+l_2}}{(t_1^2+t_2^2)^{\frac{9}{4}s}}\cdot F_q(t_1,t_2,s), d^\times t_1d^\times t_2.
    \end{equation*}
    Clearly the above is a linear combination of
    \begin{equation*}
       \int_0^\eps\int_0^\eps W_{v,u}(\mtrthr{t_1^3t_2^3}{}{}{}{t_2^3}{}{}{}{1})\cdot\frac{t_1^{9s-6+l_1}t_2^{9s-6+l_2}}{(t_1^2+t_2^2)^{\frac{9}{4}s}}\cdot r^{\frac{9}{2}s+6k}a_k(\theta) d^\times t_1d^\times t_2.
        \end{equation*}
        and
        \begin{equation}\label{Eq: 25}
         \int_0^\eps\int_0^\eps W_{v,u}(\mtrthr{t_1^3t_2^3}{}{}{}{t_2^3}{}{}{}{1})\cdot\frac{t_1^{9s-6+l_1}t_2^{9s-6+l_2}}{(t_1^2+t_2^2)^{\frac{9}{4}s}}\cdot r^{-\frac{9}{2}s+6k+3}b_k(\theta), d^\times t_1d^\times t_2.
        \end{equation}
        The proofs of the meromorphic continuation for above two integrals are almost the same. We only establish the meromorphic continuation for the second one in detail, because the second one is slightly more complicated since the denominator $(t_1^2+t_2^2)^{\frac{9}{4}s} = r^{\frac{9}{2}s}$ can not be canceled out. We note that there are only finitely many $l_1,l_2,k$. They only depend on the vertical strip Re(s)$\in(a,b)$, not on the choice of $\xi_1,\xi_2$.

        The term $t_1^{-3\eta_u}t_2^{-3\xi_u}(\ln t_1)^{r_1}(\ln t_2)^{r_2}f_0(v,u)$ contributes an integral
        \begin{equation}\label{Eq: 15}
           \int_0^\eps\int_0^\eps t_1^{-3\eta_u}t_2^{-3\xi_u}(\ln t_1)^{r_1}(\ln t_2)^{r_2}\frac{t_1^{9s-6+l_1}t_2^{9s-6+l_2}}{(t_1^2+t_2^2)^{\frac{9}{2}s}}\cdot          r^{6k+3}b_k(\theta) d^\times t_1d^\times t_2\cdot f_0(v,u)
        \end{equation}
        in (\ref{Eq: 25}). We aim to show that the above integral has a meromorphic continuation for nonnegative integers $l_1,l_2,k$ and any bounded smooth functions $b_k(\theta)$. Because the integrand is not factorizable in the Cartesian coordinate, we are facing a situation slightly more complicated than that in \cite{Sou}, yet this complication still can be resolved in the polar coordinate.
        \begin{equation*}
            t_1 = r\cos\theta, t_2 = r\sin\theta.
        \end{equation*}
       After expanding polynomials $(\ln t_1)^{r_1} = (\ln r+\ln \cos\theta)^{r_1}$ and $(\ln t_2)^{r_2} = (\ln r+\ln \sin\theta)^{r_2}$, we only need to prove that the following integral has a meromorphic continuation in $s$
        \begin{equation}\label{eq: update elementary integral}
           \iint_{D} r^{9s+a_1}(\ln r)^{a_2}\cdot(\cos\theta)^{9s+b_1}(\ln \cos\theta)^{b_2}\cdot(\sin\theta)^{9s+c_1}(\ln\cos\theta)^{c_2}f(\theta)drd\theta,
        \end{equation}
        where $D$ is a square $[0,\eps]\times[0,\eps]$, $a_1,b_1,c_1$ are polynomials of $u$ of degree 1, $a_2,b_2,c_2$ are nonnegative integers, and $f(\theta)$ is a bounded smooth function in $\theta$. But this is very elementary and the proof is based on integration by parts. For details, see Appendix \ref{appendix: elementary integral}.

        The term $t_1^{-3\eta_u}(\ln t_1)^{r_1}(\ln t_2)^{r_2}f_2(\ln t_2, v, u)$ contributes an integral
        \begin{equation}\label{Eq: 22}
           \int_0^\eps\int_0^\eps (t_1)^{-3\eta_u}(\ln t_1)^{r_1}(\ln t_2)^{r_2}f_2(\ln t_2,v,u)\frac{t_1^{9s-6}t_2^{9s-6}}{(t_1^2+t_2^2)^{\frac{9}{2}s}}\cdot          t_1^{l_1}t_2^{l_2}r^{6k+3}b_k(\theta) d^\times t_1d^\times t_2
        \end{equation}
        in (\ref{Eq: 25}).
        We change variable $t_2 \mapsto t_1t_2$, then (\ref{Eq: 22}) is equal to a finite linear combination of
        \begin{equation}\label{Eq: 23}
           \begin{aligned}
           &\int_0^\eps\int_0^{\epsilon t_1} t_1^{9s-9+l_1+l_2+6k-3\eta_u}(\ln t_1)^{r_1'}(\ln t_2)^{r_2'}f_2(\ln t_1t_2,v,u)\\
           &\quad\cdot\frac{t_2^{9s-6+l_2}}{(t_2^2+1)^{\frac{1}{2}(9s-6k-3)}}\cdot          b_k(\arctan{t_2}) d^\times t_2d^\times t_1.
           \end{aligned}
        \end{equation}
        By the estimate of $f_2$, we get
        \begin{equation*}
            \abs{f_2(\ln t_1t_2,v,u)}\leq (t_1t_2)^{-3\text{Re}(\xi_2)}h_2(\ln t_1t_2)q'(v),
        \end{equation*}
        for some polynomial $h_2$ and a continuous seminorm $q'$. As we only have finitely many $l_1,l_2,k$ which only depend on $(a,b)$, if we require the exponent $\Re(\xi_2)$ to be sufficiently negative so that it will beat the exponents of $t_1$, $t_2$ in (\ref{Eq: 23}), then the integral (\ref{Eq: 23}) converges absolutely in the vertical strip $\text{Re}(s)\in (a,b)$. Thus it defines a holomorphic function in $s$ and $u$ which is continuous in $v$. The contribution of the term $t_2^{-3\xi_u}(\ln t_1)^{r_1}(\ln t_2)^{r_2}f_1(\ln t_1, v,u)$ can be analyzed in the same way.

        The term $(\ln t_1)^{r_1}(\ln t_2)^{r_2}f_3(\ln t_1, \ln t_2,v)$ contributes an integral
        \begin{equation}\label{Eq: 16}
           \int_0^\eps\int_0^\eps (\ln t_1)^{r_1}(\ln t_2)^{r_2}f_3(\ln t_1,\ln t_2,v,u)\frac{t_1^{9s-6+l_1}t_2^{9s-6+l_2}}{(t_1^2+t_2^2)^{\frac{9}{2}s}}\cdot          r^{6k+3}b_k(\theta) d^\times t_1d^\times t_2
        \end{equation}
        in (\ref{Eq: 25}). By the estimate of $f_3$, if we choose Re$(\xi_1)$, Re$(\xi_2)$ so negative that the exponent -3Re$(\xi_1)$ and -3Re$(\xi_2)$ in the estimate of $f_3$ can beat all the exponents in the integrand of $(\ref{Eq: 16})$, then the integral $(\ref{Eq: 16})$ is holomorphic in $s,u$ and continuous in $v$. In the end, after taking the $\BC(u)$-linear combination, we finally prove that the local integrals are meromorphic in $s, u$.

        To summarize all of the above, we proved that $C(W_{v,u},f_s)$ has a meromorphic continuation in $s,u$ and continuous in $v$ under the Fr\'{e}chet topology of $V_{\pi_u}$. The continuity on the second variable is straight forward. Because $f_s$ only affects $F(t_1,t_2,s)$. The coefficients $a_k(\theta)$, $b_k(\theta)$ depend on $f_s$. They are continuous with respect to $f_s$ when $f_s$ runs in the Fr\'{e}chet space $V_{\rho_s}$.  So, if a sequence $f_{s,k}\rightarrow 0$ in $V_{\rho_s}$, then all $a_k(\theta)$, $b_k(\theta)$ tend to zero. This implies the meromorphic continuation of $C(W_{v,u},f_s)$ is continuous on $V_{\rho_s}$.
    \end{proof}

    \begin{rk}
        To end this Section, we remark that the proof of the complex case proceeds almost in the same way as the real case, the only difference comes from the $dz$-integral. In the complex case, the $dz$-integral is the Jacquet integral of a principal series of $\SL_2(\BC)$.
    \end{rk}

%%%%%%%%%%%%%%%%%%%%%%%%%%%%%%%%%%%%%%%%%%%%%%%%%%%%%%%%%%%%%%%%%%%%%%%%%%%%%%%%%%%%%%%%%%%%%%%%%%%%%%%%%%%%%%%%%%%%%%%%%%%%%%%%%%%%%%%%%%%%%%

\section{The Uniqueness Theorem (Archimedean Case)}\label{section: uniqueness thm}

%%%%%%%%%%%%%%%%%%%%%%%%%%%%%%%%%%%%%%%%%%%%%%%%%%%%%%%%%%%%%%%%%%%%%%%%%%%%%%%%
   Let us first quote a Lemma in \cite{Gin}.
    \begin{lemma}\label{Lemma: Double Coset Decompostion of G2}
       \begin{enumerate}
         \item $\lvert P\quo \RG_2/\SL_3\rvert = 2$, and we can take representatives $e$ and $\gamma = x_{-(\alpha+\beta)}(-1)w_\beta$.
         \item \begin{equation*}
                       \SL_3^\gamma := \SL_3\cap \gamma^{-1} P\gamma = \bigset{n_2\mtrthr{a}{}{}{}{1}{}{}{}{a^{-1}}}{n_2\in N_2},
               \end{equation*}
               where the subgroup $N_2$ consists matrices of the form \eqref{eq: group N2}.
         \end{enumerate}
    \end{lemma}
   Now we are ready to prove Theorem \ref{Theorem: Uniqueness Theorem}.

   \begin{proof}[Proof of Theorem \ref{Theorem: Uniqueness Theorem}]
   Given two Fr\'{e}chet representations of $\pi_1$ and $\pi_2$ of a Lie group $G$, we denote by $\Bil_G(\pi_1,\pi_2)$ the space of $G$-equivariant continuous bilinear form on $V_{\pi_1}\hat{\otimes} V_{\pi_2}$. The dimension of $\Bil_G(\pi_1,\pi_2)$ is called the intertwining number between representations $\pi_1$ and $\pi_2$. By the Reciprocity Law (\cite[Theorem 6;4]{BruhatThesis}), we get
   \begin{equation*}
      \dim \Bil_{\SL_3(F)}(\pi,\quad \Ind_{P(F)}^{\RG_2(F)}\delta_P^{s-\frac{1}{2}}\Big\lvert_{\SL_3(F)}) = \dim \Bil_{\RG_2(F)} (\Ind_{\SL_3(F)}^{\RG_2(F)}\pi,\quad \Ind_{P(F)}^{\RG_2(F)}\delta_P^{s-\frac{1}{2}}).
   \end{equation*}
    Let $\SL_3(F)$ act on the flag variety $P(F)\quo\RG_2(F)$ on the right. By Lemma \ref{Lemma: Double Coset Decompostion of G2}, there are exactly two orbits with representatives $e$ and $\gamma = x_{-(\alpha+\beta)}(-1)w_\beta$. Then by \cite[Theorem 6;3]{BruhatThesis}, we have an estimate:
   \begin{equation}\label{eq: Bruhat estimate}
       \begin{aligned}
          \dim \Bil_{\RG_2(F)} (\Ind_{\SL_3(F)}^{\RG_2(F)}\pi, \Ind_{P(F)}^{\RG_2(F)}\delta_P^{s-\frac{1}{2}})
       \leq \Big(\sum_{n=0}^{+\infty} i(\delta_p^{s-\frac{1}{2}},\pi, e,n)\Big) + i(\delta_P^{s-\frac{1}{2}},\pi,\gamma,0),
       \end{aligned}
   \end{equation}
   where $i(\delta_p^{s-\frac{1}{2}},\pi,\Omega_e,n)$ is the intertwining number between the two representations of $$H_e :=P(F)\cap \SL_3(F).$$ One is
   \begin{equation*}
       \delta_P^{s-\frac{1}{2}}\hat{\otimes}\pi
   \end{equation*}
   and the other one is some finite dimensional representation $\Lambda_n$ coming from transversal derivatives. The second term $i(\delta_P^{s-\frac{1}{2}},\pi,\gamma,0)$ on the RHS of \eqref{eq: Bruhat estimate} is the intertwining number between two representations of $H_\gamma := \SL_3(F)\cap \gamma^{-1}P(F)\gamma$:
   \begin{equation*}
      \delta_\gamma: h\mapsto \delta_P^{s-\frac{1}{2}}(\gamma h\gamma^{-1}) \qquad \text{and }\quad \pi\big\lvert_{H_\gamma}.
   \end{equation*}

   We claim that for all $n$, $i(\delta_p^{s-\frac{1}{2}},\pi, e,n) = 0$. We know that any eigenvalue of a unipotent matrix on a finite dimensional vector space must be $1$. Suppose that $i(\delta_p^{s-\frac{1}{2}},\pi,\Omega_e,n) \neq 0$ for some $n$. Since $H_e$ contains all standard upper triangular unipotent matrices in $\SL_3(F)$, the generic character must be trivial. We get a contradiction.

   As for the intertwining number $i(\delta_P^{s-\frac{1}{2}},\pi,\gamma,0)$, we apply Casselman's Subrepresentation Theorem \cite[Section 3.8.3]{Wal1} and assume that $\pi$ is a quotient of a principal series $\Ind_{B_{\SL_3}}^{\SL_3(F)}\sigma$. Then
   \begin{equation}\label{eq: update 04}
      \begin{aligned}
         \dim \Bil_{H_\gamma}(\pi\big\lvert_{H_\gamma}, \delta_\gamma) &= \dim \Bil_{\SL_3(F)}(\pi,\quad \Ind_{H_\gamma}^{\SL_3(F)}\delta_\gamma)\\
                                                                 &\leq \dim \Bil_{\SL_3(F)}(\Ind_{B_{\SL_3}}^{\SL_3(F)}\sigma,\quad \Ind_{H_\gamma}^{\SL_3(F)}\delta_\gamma).\\
      \end{aligned}
   \end{equation}
   The RHS of \eqref{eq: update 04} can be estimated via Bruhat's theory as well. By Lemma \ref{Lemma: Double Coset Decompostion of G2}, we have
   $$H_\gamma = \SL_3(F)\cap \gamma^{-1}P(F)\gamma = \Big\{\mtrthr{1}{-x}{z}{}{1}{x}{}{}{1}\mtrthr{a}{}{}{}{1}{}{}{}{a^{-1}}\Big\}.$$
   The orbit space which we consider this time is $B_{\SL_3}\quo \SL_3(F)/H_\gamma$. There are twelve orbits. We can choose the representatives $y_j=w_j\mtrthr{1}{1}{}{}{1}{}{}{}{1}$ and $w_j$, where $w_j$ are Weyl elements of $\SL_3(F)$ listed below:
   \begin{equation*}
      \begin{aligned}
      &w_1=I, &w_2=\mtrthr{1}{}{}{}{}{1}{}{-1}{}, \quad&w_3=\mtrthr{}{1}{}{-1}{}{}{}{}{1},\\
      &w_4=\mtrthr{}{1}{}{}{}{1}{1}{}{}, &w_5=\mtrthr{}{}{1}{1}{}{}{}{-1}{}, \quad&w_6=\mtrthr{}{}{1}{}{1}{}{-1}{}{}.
      \end{aligned}
   \end{equation*}
  Then
   \begin{equation*}
      \text{dim Bil}_{\SL_3(F)}(\Ind_{B_{SL_3}}^{\SL_3(F)}\sigma,\quad \Ind_{H_\gamma}^{\SL_3(F)}\delta_{\gamma})
   \end{equation*}
   is bounded by
   \begin{equation*}
      \sum_n\sum_{j=0}^6 i(\sigma, \delta_\gamma,y_j,n)+\sum_n\sum_{i=0}^6 i(\sigma, \delta_\gamma,w_i,n).
   \end{equation*}
   Here, for $i=1,2,3,4,5,6$, we can show by a simple matrix computation that $H_i=B_{\SL_3}\cap w_i^{-1}H_\gamma w_i$ contains a nontrivial diagonal subgroup. Similarly, for $j=1,2,3,4,5$, $H_j^y=B_{\SL_3}\cap y_j^{-1}H_\gamma y_j$ contains a nontrivial one dimensional abelian subgroup $A(j)$. To be more precise,
   \begin{equation*}
      \begin{aligned}
      &A(1)=\bigset{y_1^{-1} \mtrthr{a}{a-1}{}{}{1}{}{}{}{a^{-1}} y_1}{a\in F^\times};\\
      &A(2)=\bigset{y_2^{-1} \mtrthr{a}{a-1}{}{}{1}{}{}{}{a^{-1}} y_2}{a\in F^\times};\\
      &A(3)=\bigset{y_3^{-1} \mtrthr{1}{}{z}{}{1}{}{}{}{1} y_3}{z\in F};\\
      &A(4)=\bigset{\mtrthr{1}{\frac{a-1}{a}}{}{}{a^{-1}}{}{}{}{a}}{a\in F^\times};\\
      &A(5)=\bigset{\mtrthr{a^{-1}}{}{}{}{a}{-a+1}{}{}{1}}{a\in F^\times}.
      \end{aligned}
   \end{equation*}
   For each above abelian subgroup $A(j)$, we can choose a generator $d_j$. By Bruhat theory, $i(\sigma, \delta_\gamma,y_j,n)$ is the intertwining number between the representation
   \begin{equation}
      \sigma\otimes (\delta_\gamma)^{y_j}
   \end{equation}
   of $H_j^y$ and some finite dimensional representation $\tilde{\Lambda}_n$ of $H_j^y$ coming from transversal derivatives. Here $(\delta_\gamma)^{y_j}$ is the representation obtained from twisting $\delta_\gamma$ by ${y_j}$. If $i(\sigma, \delta_\gamma,y_j,n)$ is nonzero for some $n$, then $\sigma\otimes (\delta_\gamma)^{y_j}$ is a one dimensional subrepresentation of $\tilde{\Lambda}_n$.  Hence
   \begin{equation}\label{Eq£º19}
       \sigma\otimes (\delta_\gamma)^{y_j}(d_j) = \chi(d_j)
   \end{equation}
   for some character $\chi$. The above is in fact an equation of $s$ which only has at most countably many solutions. These solutions form a discrete set. The same argument applies to $i(\sigma, \delta_\gamma,w_i,n)$.

   Thus so far, we have at most countably many finite dimensional representations $\tilde{\Lambda_n}$. For each $\tilde{\Lambda}_n$, we only have finitely many one-dimensional subrepresentations. Thus there exists a discrete, at most countable subset $S$ of $\BC$ such that whenever $s\notin S$,
   \begin{equation*}
          i(\sigma, \delta_\gamma,y_j,n) = i(\sigma, \delta_\gamma, w_i,n) = 0,
   \end{equation*}
   where $j=1,2,3,4,5,$ and $i = 1,2,3,4,5,6$. As for $H_6^y$, it is in fact a trivial subgroup, so the above equation (\ref{Eq£º19}) is in fact an identity. The big orbit $B_{\SL_3}y_6H_\gamma$ contributes one in
   \begin{equation*}
        \dim \Bil_{\SL_3(F)}(\pi,\quad \Ind_{P(F)}^{\RG_2(F)}\delta_P^{s-\frac{1}{2}}\Big\lvert_{\SL_3(F)}).
   \end{equation*}

   \end{proof}

%%%%%%%%%%%%%%%%%%%%%%%%%%%%%%%%%%%%%%%%%%%%%%%%%%%%%%%%%%%%%%%%%%%%%%%%%%%%%%%%%%%%%%%%%%%%%%%%%%%%%%%%%%%%%%%%%%%%%%%%%%%%%%%%%%%%%%%%%%%%%%%%%%%

\appendix

%%%%%%%%%%%%%%%%%%%%%%%%%%%%%%%%%%%%%%%%%%%%%%%%%%%%%%%%%%%%%%%%%%%%%%%%%%%%%%%%%%%%%%%%%%%%%%%%%%%%%%%%%%%%%%%%%%%%%%%%%%%%%%%%%%%

    \section{Some Elementary Integrals}\label{appendix: elementary integral}

%%%%%%%%%%%%%%%%%%%%%%%%%%%%%%%%%%%%%%%%%%%%%%%%%%%%%%%%%%%%%%%%%%%%%%%%%%%%%%%%%%%%%%%%%%%%%%%%%%%%%%%%%%%%%%%%%%%%%%%%%%%%%%%%%
    In this Appendix, we will prove the meromorphic continuation of the elementary integral \eqref{eq: update elementary integral} which appeared in the proof of Theorem \ref{Theorem: Meromorphic Continuation}. Let $D_\eps$ be the unit square $[0,\eps]\times [0,\eps]$ on the $xy$-plane. We use polar coordinate
    \begin{equation*}
       x=r\cos\theta, y=r\sin\theta.
    \end{equation*}
    Our goal is to prove
    \begin{prop}\label{Prop: Technical integral in polar coor}
       Let $f(\theta, s)$ be a smooth periodic function of $\theta$ on $[0,2\pi]$, which is holomorphic in $s$ in some right half plane $\text{Re}(s)\geq s_0$, then for any complex numbers $a_1, b_1, c_1$, and any nonnegative integers $a_2,b_2,c_2$, the following integral
       \begin{equation}\label{eq: update 03}
          \iint_{D_\eps} r^{s+a_1}(\ln r)^{a_2}\cdot(\cos\theta)^{s+b_1}(\ln \cos\theta)^{b_2}\cdot(\sin\theta)^{s+c_1}(\ln\cos\theta)^{c_2}f(\theta,s)drd\theta
       \end{equation}
       converges absolutely when $\text{Re}(s)$ is sufficiently large, and it has a meromorphic continuation to the same right half plane $\text{Re}(s)\geq s_0$. Moreover, if $a_1,b_1,c_1$ are polynomials of a complex variable $u$ of degree 1, then the meromorphic continuation of \eqref{eq: update 03} (in $s$) is also meromorphic in $u$.
    \end{prop}
    Clearly the meromorphic continuation of \eqref{eq: update elementary integral} is a direct consequence of Proposition \eqref{Prop: Technical integral in polar coor}.
    \begin{lemma}\label{Lemma: Quasi-Beta Funciton}
       Let $f(\theta, s_1,s_2)$ be a smooth periodic function of $\theta$ on $[0,2\pi]$, holomorphic in some right half space $\set{(s_1,s_2)\in \BC^2}{\text{Re}(s_1),\text{Re}(s_2)\geq s_0}$. Then for any complex numbers $b_1, c_1$, and nonnegative integers $b_2,c_2$ the following integral
       \begin{equation}\label{Eq: 17}
          \int_0^{\frac{\pi}{2}} (\cos\theta)^{s_1+b_1}(\ln \cos\theta)^{b_2}\cdot(\sin\theta)^{s_2+c_1}(\ln\cos\theta)^{c_2}f(\theta,s_1,s_2)d\theta
       \end{equation}
       converges absolutely when $\text{Re}(s_1), \text{Re}(s_2)$ are sufficiently large, and it has a meromorphic continuation to the same right half space $\set{(s_1,s_2)\in \BC^2}{\text{Re}(s_1),\text{Re}(s_2)\geq s_0}$.
    \end{lemma}
    \begin{proof}
       It is clear that the integral \eqref{Eq: 17} converges when $\text{Re}(s_1), \text{Re}(s_2)$ are sufficiently large. When $\text{Re}(s_1), \text{Re}(s_2)$ are large, by integration by parts, we have
       \begin{equation*}
          \begin{aligned}
          &\int_0^{\frac{\pi}{2}} (\cos\theta)^{s_1+b_1+2}(\ln \cos\theta)^{b_2}\cdot(\sin\theta)^{s_2+c_1}(\ln\cos\theta)^{c_2}f(\theta,s_1,s_2)d\theta\\
          =&\frac{1}{s_2+c_1+1}\int_0^{\frac{\pi}{2}}(s_1+b_1+1)(\sin\theta)^{s_2+c_1+2}(\cos\theta)^{s_1+b_1}(\ln\cos\theta)^{b_2}(\ln\sin\theta)^{c_2}f(\theta,s_1,s_2)d\theta\\
           &+\frac{1}{s_2+c_1+1}\int_0^{\frac{\pi}{2}}(\sin\theta)^{s_2+c_1+2}(\cos\theta)^{s_1+b_1}(\ln\cos\theta)^{b_2-1}(\ln\sin\theta)^{c_2}f(\theta,s_1,s_2)d\theta\\
           &-\frac{1}{s_2+c_1+1}\int_0^{\frac{\pi}{2}}(\sin\theta)^{s_2+c_1}(\cos\theta)^{s_1+b_1+2}(\ln\cos\theta)^{b_2}(\ln\sin\theta)^{c_2-1}f(\theta,s_1,s_2)d\theta\\
           &-\frac{1}{s_2+c_1+1}\int_0^{\frac{\pi}{2}}(\sin\theta)^{s_2+c_1+1}(\cos\theta)^{s_1+b_1+1}(\ln\cos\theta)^{b_2}(\ln\sin\theta)^{c_2}\frac{\partial f}{\partial \theta}(\theta,s_1,s_2)d\theta.\\
          \end{aligned}
       \end{equation*}
       Thus,
       \begin{equation*}
          \begin{aligned}
          &\frac{s_1+b_1+1}{s_2+c_1+1}\int_0^{\frac{\pi}{2}} (\cos\theta)^{s_1+b_1}(\ln \cos\theta)^{b_2}\cdot(\sin\theta)^{s_2+c_1}(\ln\cos\theta)^{c_2}f(\theta,s_1,s_2)d\theta\\
          &=(1+\frac{s_1+b_1+1}{s_2+c_1+1})\int_0^{\frac{\pi}{2}} (\cos\theta)^{s_1+b_1+2}(\ln \cos\theta)^{b_2}\cdot(\sin\theta)^{s_2+c_1}(\ln\cos\theta)^{c_2}f(\theta,s_1,s_2)d\theta\\
          &-\frac{1}{s_2+c_1+1}\int_0^{\frac{\pi}{2}}(\sin\theta)^{s_2+c_1+2}(\cos\theta)^{s_1+b_1}(\ln\cos\theta)^{b_2-1}(\ln\sin\theta)^{c_2}f(\theta,s_1,s_2)d\theta\\
           &+\frac{1}{s_2+c_1+1}\int_0^{\frac{\pi}{2}}(\sin\theta)^{s_2+c_1}(\cos\theta)^{s_1+b_1+2}(\ln\cos\theta)^{b_2}(\ln\sin\theta)^{c_2-1}f(\theta,s_1,s_2)d\theta\\
           &+\frac{1}{s_2+c_1+1}\int_0^{\frac{\pi}{2}}(\sin\theta)^{s_2+c_1+1}(\cos\theta)^{s_1+b_1+1}(\ln\cos\theta)^{b_2}(\ln\sin\theta)^{c_2}\frac{\partial f}{\partial \theta}(\theta,s_1,s_2)d\theta.\\
          \end{aligned}
       \end{equation*}
       Similarly, there is another integration by parts which raises the exponents of $\sin$ by 2. So, whenever we apply integration by parts, either the exponents of $\sin$ or $\cos$ increase, or the exponents of $\ln\sin$ or $\ln\cos$ decrease. Hence it suffices to prove that the following integral has a meromorphic continuation
       \begin{equation}\label{Eq: 18}
          \int_0^{\frac{\pi}{2}} (\cos\theta)^{s_1+b_1}(\sin\theta)^{s_2+c_1}f(\theta,s_1,s_2)d\theta.
       \end{equation}
       As above, by integration by parts, we get
       \begin{equation*}
          \begin{aligned}
          &\frac{s_1+b_1+1}{s_2+c_1+1}\int_0^{\frac{\pi}{2}}(\cos\theta)^{s_1+b_1}(\sin\theta)^{s_2+c_1}f(\theta,s_1,s_2)d\theta\\
          =&(1+\frac{s_1+b_1+1}{s_2+c_1+1})\int_0^{\frac{\pi}{2}}(\cos\theta)^{s_1+b_1+2}(\sin\theta)^{s_2+c_1}f(\theta,s_1,s_2)d\theta\\
          +&(\frac{1}{s_2+c_1+1})\int_0^{\frac{\pi}{2}}(\cos\theta)^{s_1+b_1+1}(\sin\theta)^{s_2+c_1+1}\frac{\partial f}{\partial\theta}(\theta,s_1,s_2)d\theta.\\
          \end{aligned}
       \end{equation*}
       After doing integration by parts as many times as we need, the exponents of $\sin\theta$ and $\cos\theta$ are nonnegative when $\text{Re}(s_1),\text{Re}(s_2)\geq s_0$, this shows (\ref{Eq: 18}), and hence (\ref{Eq: 17}), has a meromorphic continuation.
    \end{proof}
    \begin{rk}
       We can also show that the following integrals
       \begin{equation*}
          \int_0^{\frac{\pi}{4}} (\cos\theta)^{s+b_1}(\ln \cos\theta)^{b_2}\cdot(\sin\theta)^{s+c_1}(\ln\cos\theta)^{c_2}f(\theta,s)d\theta
       \end{equation*}
        and
        \begin{equation*}
          \int_{\frac{\pi}{4}}^{\frac{\pi}{2}} (\cos\theta)^{s+b_1}(\ln \cos\theta)^{b_2}\cdot(\sin\theta)^{s+c_1}(\ln\cos\theta)^{c_2}f(\theta,s)d\theta
       \end{equation*}
        also have meromorphic continuation. The proof is the same as the proof of Lemma \ref{Lemma: Quasi-Beta Funciton}. We only need to change the upper and lower bounds of the integrals.
    \end{rk}
    Now we return to prove Proposition \ref{Prop: Technical integral in polar coor}. For simplicity, we only prove it for $\eps=1$, the proof for general $\eps$ is exactly the same (only notationally more complicated). In the following proof, to simplify notations, we write $D$ for the unit square in the first quadrant instead of $D_1$.
    \begin{proof}[Proof of Proposition \ref{Prop: Technical integral in polar coor}]
       Set $D_1$ to be the quarter of unit disk in the first quadrant, $D_2=D-D_1$. Then
       \begin{equation*}
          \begin{aligned}
          &\iint_{D_1} r^{s+a_1}(\ln r)^{a_2}\cdot(\cos\theta)^{s+b_1}(\ln \cos\theta)^{b_2}\cdot(\sin\theta)^{s+c_1}(\ln\cos\theta)^{c_2}f(\theta,s)drd\theta\\
          =&\int_0^1 r^{s+a_1}(\ln r)^{a_2}dr\cdot\int_{0}^{\frac{\pi}{2}}(\cos\theta)^{s+b_1}(\ln \cos\theta)^{b_2}\cdot(\sin\theta)^{s+c_1}(\ln\cos\theta)^{c_2}f(\theta,s)d\theta,
          \end{aligned}
       \end{equation*}
       hence by Lemma \ref{Lemma: Quasi-Beta Funciton}, the above integral has a meromorphic continuation. On the other hand,
       \begin{equation*}
          \begin{aligned}
          &\iint_{D_2}r^{s+a_1}(\ln r)^{a_2}\cdot(\cos\theta)^{s+b_1}(\ln \cos\theta)^{b_2}\cdot(\sin\theta)^{s+c_1}(\ln\cos\theta)^{c_2}f(\theta,s)drd\theta\\
          =&\int_0^{\frac{\pi}{4}}\Big(\int_1^{\frac{1}{\cos\theta}}r^{s+a_1}(\ln r)^{a_2}dr\Big)(\cos\theta)^{s+b_1}(\ln \cos\theta)^{b_2}\cdot(\sin\theta)^{s+c_1}(\ln\cos\theta)^{c_2}f(\theta,s)d\theta\\
          +&\int_{\frac{\pi}{4}}^{\frac{\pi}{2}}\Big(\int_1^{\frac{1}{\sin\theta}}r^{s+a_1}(\ln r)^{a_2}dr\Big)(\cos\theta)^{s+b_1}(\ln \cos\theta)^{b_2}\cdot(\sin\theta)^{s+c_1}(\ln\cos\theta)^{c_2}f(\theta,s)d\theta.
          \end{aligned}
       \end{equation*}
       Applying integration by parts to $dr$-integral successively, we can show that the $dr$-integral in the first term on the right hand side is a linear combination of $(\cos\theta)^{b_1'}(\ln\cos\theta)^{b_2'}$, with coefficients in the field of rational functions $\BC(s)$. So by Lemma \ref{Lemma: Quasi-Beta Funciton} and the remark below it, the first term on the right hand side has a meromorphic continuation. Similarly, the second term also has a meromorphic continuation. Hence we finish the proof of Proposition \ref{Prop: Technical integral in polar coor}.
    \end{proof}
%%%%%%%%%%%%%%%%%%%%%%%%%%%%%%%%%%%%%%%%%%%%%%%%%%%%%%%%%%%%%%%%%%%%%%%%%%%%%%%%

%%%%%%%%%%%%%%%%%%%%%%%%%%%%%%%%%%%%%%%%%%%%%%%%%%%%%%%%%%%%%%%%%%%%%%%%%%%%%%%%%%%%%%%%%%%%%%%%%%%%%%%%%%%
\bibliographystyle{elsarticle-num}
    
\end{document}